\IfFileExists{aux/bibliography.bib}{\immediate\write18{biber \jobname}}{}

\documentclass{amsart}
\usepackage{amssymb}
\usepackage{tikz-cd}
\usepackage{mathbbol} %
\usepackage{csquotes}
\usepackage{adjustbox}
\usepackage{todo}

\usepackage[
    backend=biber,
    style=alphabetic,
    backref=true,
    url=false,
    doi=false,
    isbn=false,
    eprint=false
]{biblatex}

\renewbibmacro{in:}{}  %
\AtEveryBibitem{\clearfield{pages}}  %

\DeclareFieldFormat{title}{\myhref{\mkbibemph{#1}}}
\DeclareFieldFormat
[article,inbook,incollection,inproceedings,patent,thesis,unpublished]
{title}{\myhref{\mkbibquote{#1\isdot}}}

\newcommand{\doiorurl}{%
	\iffieldundef{url}
	{\iffieldundef{eprint}
		{}
		{http://arxiv.org/abs/\strfield{eprint}}}
	{\strfield{url}}%
}

\newcommand{\myhref}[1]{%
	\ifboolexpr{%
		test {\ifhyperref}
		and
		not test {\iftoggle{bbx:eprint}}
		and
		not test {\iftoggle{bbx:url}}
	}
	{\href{\doiorurl}{#1}}
	{#1}%
}

\usepackage[
    bookmarks=true,
    linktocpage=true,
    bookmarksnumbered=true,
    breaklinks=true,
    pdfstartview=FitH,
    hyperfigures=false,
    plainpages=false,
    naturalnames=true,
    colorlinks=true,
    pagebackref=false,
    pdfpagelabels
]{hyperref}

\hypersetup{
	colorlinks,
	citecolor=blue,
	filecolor=blue,
	linkcolor=blue,
	urlcolor=blue
}

\usepackage[capitalize, noabbrev]{cleveref}
\let\subsectionSymbol\S
\crefname{subsection}{\subsectionSymbol\!}{subsections}

\calclayout

\makeatletter
\@namedef{subjclassname@2020}{%
	\textup{2020} Mathematics Subject Classification}
\makeatother

\theoremstyle{plain}
\newtheorem{thm}{Theorem}[section]
\Crefname{thm}{Theorem}{Theorems}
\newtheorem{cor}[thm]{Corollary}
\newtheorem{lem}[thm]{Lemma}
\newtheorem{prop}[thm]{Proposition}
\newtheorem*{thm*}{Theorem}
\theoremstyle{definition}
\newtheorem{defi}[thm]{Definition}
\newtheorem{rem}[thm]{Remark}
\Crefname{rem}{Remark}{Remarks}
\newtheorem{example}[thm]{Example}

\newcommand{\N}{\mathbb{N}}
\newcommand{\Z}{\mathbb{Z}}
\newcommand{\F}{\mathbb{F}}
\newcommand{\R}{\mathbb{R}}
\newcommand{\HS}{\mathrm{HS}}

\newcommand{\HLC}{\mathrm{LHS}}

\newcommand{\LC}{\mathrm{LC}}

\newcommand{\piLC}{\mathrm{LC}^\infty}
\newcommand{\HE}{\mathbb{H}^d}
\renewcommand{\H}{H}
\DeclareMathOperator{\interior}{int}
\DeclareMathOperator{\im}{im}
\DeclareMathOperator{\coker}{coker}
\DeclareMathOperator{\rad}{rad}
\newcommand{\CH}{\check{H}}
\newcommand{\cp}{\check{p}}
\newcommand{\cc}{\check{c}}
\newcommand{\calpha}{\check{\alpha}}
\newcommand{\comega}{\check{\omega}}
\newcommand{\PFD}{\mathrm{PFD}}
\DeclareMathOperator{\Nrv}{Nrv}
\DeclareMathOperator{\Cov}{Cov}
\DeclareMathOperator{\Vietoris}{Vtr}
\DeclareMathOperator{\Balls}{BallCov}
\DeclareMathOperator*{\colim}{colim}
\DeclareMathOperator{\crit}{crit}
\newcommand{\mulDom}{\mathcal{E}}
\newcommand{\Top}{\mathsf{Top}}
\newcommand{\Vect}{\mathsf{Vect}}
\renewcommand{\th}{^{\mathrm{th}}}
\newcommand{\m}{\mathfrak{m}}

\newcommand{\defeq}{\stackrel{\mathrm{def}}{=}}

\title{Persistent homology for functionals}

\author[U.~Bauer]{Ulrich Bauer}
\address{U.B., Technical University of Munich}
\email{\href{mailto:mail@ulrich-bauer.org}{mail@ulrich-bauer.org}}

\author[A.~Medina-Mardones]{Anibal M. Medina-Mardones}
\address{A.M-M., Max Planck Institute for Mathematics \and University of Notre Dame}
\email{\href{mailto:ammedmar@mpim-bonn.mpg.de}{ammedmar@mpim-bonn.mpg.de}}

\author[M.~Schmahl]{Maximilian Schmahl}
\address{M.S., Heidelberg University}
\email{\href{mailto:mschmahl@mathi.uni-heidelberg.de}{mschmahl@mathi.uni-heidelberg.de}}

\date{\today}

\begin{document}

\begin{abstract}
    We introduce topological conditions on a broad class of functionals that ensure that the persistent homology modules of their associated sublevel set filtration admit persistence diagrams, which, in particular, implies that they satisfy generalized Morse inequalities.
    We illustrate the applicability of these results by recasting the original proof of the Unstable Minimal Surface Theorem given by Morse and Tompkins in a modern and rigorous framework.
\end{abstract}
	\maketitle
	
\section{Introduction}

The interplay between the critical set of a function and the topology of its domain is a cornerstone of modern mathematics.
Nowadays, when thinking about the pioneering work of Marston Morse, our first thought probably involves a differentiable function on a closed smooth manifold, but more general settings should also be considered.
Morse theory in the smooth context was masterfully presented in Milnor's famous book on the subject \cite{Milnor.1963}, where he also gave a new proof of Bott's periodicity by applying Morse theory to the energy functional of paths in a Riemannian manifold, which notably goes beyond the compact setting.
Another important example of the use of Morse's insights in an infinite-dimensional context is Floer's work on the Arnold conjecture and its many ramifications in symplectic topology, as surveyed for example in \cite{Salamon.1999}.
Morse himself worked in a very general setting, publishing in the 1930s a pair of papers \cite{Morse.1937, Morse.1940} and a monograph \cite{Morse.1938} in which he established the key results of Morse theory in the broad context defined by semi-continuous functionals on metric spaces.
He called the theory set forth in this body of work \emph{functional topology} and used it to study questions about minimal surfaces motivated by Douglas' solution to Plateau’s Problem \cite{Douglas.1931}.
In particular, Morse and Tompkins \cite{Morse.1939} used these techniques to prove a general \emph{Mountain Pass Lemma} --~an existence result for saddle points~-- applying to functions that are not necessarily continuous (\cref{thm:mountain_pass}).
From this, they deduce their \emph{Unstable Minimal Surface Theorem}, showing the existence of critical points of the Douglas functional that are not local minima (\cref{thm:unstable_minimial_surface}).
In the intervening years, this result has been reproven and generalized in several directions using various techniques, and the problem class is still an active area of research \cite{Struwe.1984,Jost.1990,Jost.1991,Montezuma.2020,Marques.2021}.

Morse's work on functional topology did not have a long lasting impact on minimal surface theory or the calculus of variations in general; possibly in part because, as expressed by Struwe:
\begin{displaycquote}[p.~82]{Struwe.1988}
	The technical complexity and the use of a sophisticated topological machinery [...] tend to make Morse--Tompkins' original paper unreadable and inaccessible for the non-specialist.
\end{displaycquote}
A similar assessment was given by Bott, who writes in \cite[p.~934]{Bott.1980} that the papers \cite{Morse.1937, Morse.1940} ``are not easy reading'' and constitute a ``tour de force'' by Morse.

The intricacies of Morse's development notwithstanding, many of his ideas have subsequently resurfaced and flourished in other domains.
In particular, in applied topology and symplectic geometry, several key insights of Morse have been independently rediscovered as part of the development of \emph{persistent homology}, a technique that provides robust and efficiently computable invariants of filtered spaces using the functorial properties of homology.
Its success in these fields has motivated a refined abstract theory of persistence that lies in the intersection of geometry, topology, and representation theory.

The homology of a filtered space is an example of what is referred to as a \emph{persistence module}, a functor to vector spaces from the real numbers considered as a poset category.
In many important cases, a persistence module $M$ admits an essentially unique decomposition into indecomposable direct summands, and the structure of this decomposition yields a complete invariant of $M$ known as its \emph{persistence diagram}.
The set of all persistence diagrams can be organized into a metric space.
This often allows to recast geometric questions about general filtered spaces in a combinatorial metric model, since the passage via the homology construction to this metric space is Lipschitz, a statement commonly known as the \emph{stability} of persistence diagrams.

The most remarkable connections between functional topology and persistence theory come from Morse's paper \cite{Morse.1940}, where he developed the theory of \emph{caps} and their \emph{spans}.
They capture much of the same information as the modern notion of persistence diagram, including concepts such as the persistence or birth and death of a homology class, although Morse's results still fall short of yielding global decompositions of persistence modules.
Morse used his theory of caps to study functionals on a metric space by analyzing the evolution of the topology of their sublevel sets.
A key tool to this end is a version of his eponymous inequalities for cap numbers, which expands their usual version in the compact and smooth setting.
In this work, using persistence diagrams, we generalize the definition of these cap numbers to persistence modules and prove the existence of Morse inequalities for a large class of them (\cref{t:inequalities}).
Our approach makes these inequalities accessible in new contexts beyond those originally covered by functional topology.

Given the importance of persistence diagrams, in particular for stating and proving Morse inequalities, our focus will then be on the study of topological properties ensuring their existence for a broad class of filtered spaces and homology constructions.
For general persistence modules, a well studied condition for the existence of persistence diagrams is \emph{q-tameness} \cite{Chazal.2016a,Chazal.2016b}, which simply states that all linear maps between different real values in the persistence module have finite rank.
This condition is satisfied by a large class of important constructions; for example, the Vietoris--Rips or \v Cech persistent homology of a totally bounded metric space is q-tame \cite{Chazal.2014}.
The motivating question can then be reformulated as asking for topological conditions on a filtered space that ensure its persistent homology to be q-tame.
We now present the answers provided in the present work.

The persistence module associated to a filtration depends on the homology construction used, which, even when agreeing on cellular spaces, need not coincide for general topological spaces.
We restrict attention to homotopy invariant functors to the category of graded vector spaces satisfying the Mayer--Vietoris property for either open or closed sets, the primary examples being singular homology or \v{C}ech homology, respectively.

The sublevel set filtration $f_{\leq t} = f^{-1}(-\infty, t]$ of a real valued function $f$ is called \emph{locally homologically small} or \emph{$\HLC$} for a given homology theory if for any $x \in X$, any neighborhood $V$ of $x$, and any pair of indices $s,t$ with $f(x) < s < t$, there is a neighborhood $U$ of $x$ with $U \subseteq V$ such that the inclusion $f_{\leq s} \cap U \hookrightarrow f_{\leq t} \cap V$ is \emph{homologically small}; that is to say, the induced map in homology has finite rank in every degree.
We say that a sublevel set filtration is \emph{compact} if all sublevel sets are compact Hausdorff spaces.
We can now state our main result (\cref{t:local connectedness implies q-tameness}):

\begin{thm*}
	If the sublevel set filtration of a function $f \colon X \to \R$ is compact and	$\HLC$, then its persistent homology is also q-tame.
\end{thm*}

\noindent We also introduce a weaker local-connectivity condition that can be used instead of $\HLC$ in the statement above if the filtration is defined by a continuous functional (\cref{c:q-tameness for continuous functions}).

To illustrate the applicability of our results we return to the original setting of the Douglas functional that motivated the development of functional topology.
This functional satisfies the hypotheses of \cref{t:local connectedness implies q-tameness}, so its associated persistent \v{C}ech homology is indeed q-tame and admits a persistence diagram.
From this, one can easily deduce the existence of an unstable minimal surface.
As it turns out, the local connectivity conditions proposed by Morse \cite{Morse.1938,Morse.1939,Morse.1940} are not sufficient for this purpose, which we illustrate by a counterexample (\cref{c:counterexample}).

\subsection*{Summary}

The primary contribution of this work consists in a modern development of the homological aspects of Morse's functional topology from the perspective of persistence theory.
We adjust several key definitions and prove stronger statements -- including a generalized version of the Morse inequalities -- in order to allow for novel uses of persistence techniques in the calculus of variations.
We provide sufficient conditions for a lower semicontinuous function to have q-tame persistent sublevel set homology, and hence to admit a persistence diagram.
As an application of these results, we correct an inaccuracy in a result by Morse, which was employed in the proof of the Unstable Minimal Surface Theorem given by Morse and Tompkins.

\subsection*{Outline}

In \cref{s:persistence} we recall the foundations of persistence theory, considering the persistent homology of a sublevel set filtration as the key example.
We present a persistence-theoretic point of view on Morse inequalities in \cref{s:inequalities}.
It generalizes both their versions in the smooth and compact setting as well as the one used in functional topology.
The main result of the present work is presented in \cref{s:connectivity}, where we define two natural notions of local-connectivity for a sublevel set filtration and show under what circumstances they imply q-tameness of its associated persistence module.
We close in \cref{s:surfaces} with a historical overview of Morse--Tompkins' application of functional topology to minimal surface theory, and we explore its relation to our results.
\cref{s:vietoris} contains a brief discussion on the definitions of Vietoris and \v{C}ech homology and their equivalence for compact metric spaces.

\section{Filtered spaces and persistence theory} \label{s:persistence}

In this section we present an overview of the theory of persistence through filtrations by sublevel sets of real-valued functions on general topological spaces.
For a detailed exposition we refer to \cite{Chazal.2016a, Oudot.2015, Polterovich.2020} and for applications of this theory in symplectic geometry to
\cite{Polterovich.2016, Usher.2016, LeRoux.2018, Shelukhin.2019}.

The ``pipeline'' of topological persistence traverses through geometry, algebra and discrete mathematics as follows:
Given a space $X$ filtered by the sublevel sets of a function $f$, the application of some homology theory in degree $n$ with coefficients in a field, or, more generally, a functor from topological spaces to vector spaces, produces a \emph{persistence module}, an algebraic object equipped with a structure theory leading in favorable cases to a powerful invariant called \emph{persistence diagram}.

These invariants are key to applications.
For example, in the next section we will see how they lead to a generalization of the classical Morse inequalities.
Much of their usefulness results from a remarkable fact known as \emph{stability}, stating that the map from functions on $X$ with the supremum norm to the set of all persistence diagrams is 1-Lipschitz with respect to a certain natural metric called \emph{bottleneck distance} \cite{Cohen-Steiner.2007}.

In more detail, consider a space $X$ and a function $f \colon X \to \R$.
Unless stated otherwise, the functions we consider need not be continuous.
We pass to filtered spaces by considering the \emph{sublevel set filtration $f_{\leq \bullet}$ of $X$ induced by $f$}, which is defined by
\begin{equation*}
f_{\leq t} = f^{-1}(-\infty, t].
\end{equation*}

For the next step in the persistence pipeline, one needs a coherent assignment of a vector space to any topological space, or, more precisely, a functor from the category $\Top$ of topological spaces to the category $\Vect$ of vector spaces over a fixed field $\F$.
The typical choices are given by \emph{homology theories}, which for now are only assumed to be $\Z$-graded families $\H = (\H_d)_{d \in \Z}$ of \emph{homotopy invariant functors}, meaning that they assign the same morphism to homotopic maps.
Of particular importance to us are \v{C}ech homology \cite[Section IX--X]{Eilenberg.1952} with coefficients in $\F$, and homology theories in the sense of Eilenberg--Steenrod \cite[Section I]{Eilenberg.1952}, such as singular homology, again with coefficients in $\F$ \cite{Eilenberg.1944}.

By definition, applying to a filtered space $\{X_t\}_{t \in \R}$ a functor from $\Top$ to $\Vect$ yields, for every $t \in \R$ a vector space~$M_t$ and for any pair~$s, t \in \R$ with~$s \leq t$, a linear map~$M_{s,t} \colon M_s \to M_t$ such that $M_{t,t}$ is the identity and the composition $M_{s,t} \circ M_{r,s}$ is equal to $M_{r,t}$ for any triple $r \leq s \leq t$.
In other words, we obtain a functor $M$ from the real numbers, considered as a poset category, to the category of vector spaces.
Such functors are called \emph{persistence modules}.
A morphism of persistence modules $\varphi \colon M \to N$ is a natural transformation, i.e., an assignment of a linear map $\varphi_t \colon M_t \to N_t$ for every $t \in \R$ making the diagram
\begin{equation*}
\begin{tikzcd}
M_{s} \arrow[r, "M_{s,t}"] \arrow[d] & M_{t} \arrow[d] \\
N_{s} \arrow[r, "N_{s,t}"] & N_{t}
\end{tikzcd}
\end{equation*}
commute for every pair $s \leq t$.

Since most examples of interest arise from applying a homology theory to a filtered space, we also consider \emph{graded persistence modules}, which are collections of persistence modules indexed by the integers.
To lighten the presentation of the theory of persistence, in what follows we will solely focus on persistence modules, omitting straightforward generalizations to their graded counterparts.

Functorial constructions defined on the category of vector spaces over the field~$\F$ can be transferred to the category of persistence modules by applying them pointwise.
For example, the kernel and cokernel of a morphism as well as the direct sum of persistence modules are well-defined.
Persistence modules that are \emph{indecomposable}, i.e., those that have only trivial direct sum decompositions, play an important role in persistence theory.
A rich family of indecomposable persistence modules is given by \emph{interval modules}, which for intervals $I \subseteq \R$ are defined by
\begin{equation*}
C(I)_t =
\begin{cases}
\F & \text{if } t \in I, \\
0 & \text{otherwise},
\end{cases}
\qquad \qquad
C(I)_{s, t} =
\begin{cases}
\operatorname{id}_{\F} & \text{if } s, t \in I, \\
0 & \text{otherwise}.
\end{cases}
\end{equation*}
These indecomposable interval modules can be used as building blocks for \emph{barcode modules}, which are direct sums of interval modules.
Given a barcode module $\bigoplus_{\lambda \in \Lambda} C(I_{\lambda})$, the associated multiset of intervals $\{I_{\lambda}\}_{\lambda \in \Lambda}$ is known as its \emph{barcode}.
By a version of the Krull--Remak--Schmidt--Azumaya Theorem \cite{Azumaya.1950} (see also \cite[Theorem 2.7]{Chazal.2016a} for a specialization to barcode modules), two isomorphic barcode modules have the same barcodes up to a choice of the index set $\Lambda$.
Thus, if a persistence module $M$ is provided with an isomorphism to a barcode module, referred to as a \emph{barcode decomposition}, the associated barcode is a complete isomorphism invariant of $M$.
Hence, understanding which persistence modules admit barcode decompositions is key.

The most commonly used existence result for barcode decompositions is due to Crawley-Boevey \cite{Crawley-Boevey.2015}.
It guarantees the existence of a barcode decomposition for any \emph{pointwise finite dimensional} ($\PFD$) persistence module, which is a persistence modules $M$ such that $M_t$ is a finite dimensional vector space for all~$t \in \R$.
Unfortunately, the $\PFD$ condition is too restrictive for many purposes.
In particular, it is unsuited for the applications of Morse and Tompkins in minimal surface theory.
An appropriate weakening of $\PFD$ for more general settings is the notion of \emph{q-tameness}.
A persistence module $M$ is q-tame if the rank of the map $M_{s,t} \colon M_s \to M_t$ is finite for all $s < t$ \cite{Chazal.2016a}.
As exemplified by the infinite product of interval modules $\prod_{n \in \N_{> 0}} C([0,1/n))$, not every q-tame persistence module admits a barcode decomposition in the above sense.
Yet, there are multiple ways to regularize q-tame persistence modules in order to obtain invariants similar to barcodes \cite{Chazal.2016a, Chazal.2016b, Schmahl.2021}.
We briefly recall the approach from \cite{Chazal.2016b}.

A persistence module $M$ is called \emph{ephemeral} if the maps $M_{s,t} \colon M_s \to M_t$ are zero for all $s < t$.
The \emph{radical} $\rad M$ of a persistence module $M$ is the unique minimal submodule of $M$ such that the cokernel of the inclusion $\rad M \hookrightarrow M$ is an ephemeral persistence module.
More explicitly, we have $(\rad M)_t = \sum_{s<t}\im M_{s,t}$.
As an example, the radical of the infinite product $\prod_{n \in \N_{> 0}} C([0,1/n))$ is the direct sum $\bigoplus_{n \in \N_{> 0}} C((0,1/n))$.
If $M$ is q-tame, then its radical admits a barcode decomposition \cite[Corollary~3.6]{Chazal.2016b},
with the associated barcode describing the isomorphism type of~$M$ ``up to ephemerals".
This can be formalized by constructing the \emph{observable category of persistence modules}, which is equivalent to the quotient of the category of persistence modules by the full subcategory of ephemeral persistence modules.
The barcode of the radical of a q-tame persistence module $M$ is then a complete invariant of $M$ in the observable category.

Intuitively, one may think of the observable category as forgetting all information in persistence modules that does not persist over a non-zero amount of time.
In certain situations, no information is lost in this process, for example if the persistence module $M$ can be assumed to be \emph{continuous from above at $s$} for all~$s \in \R$, i.e., $M_{s} \to \lim_{s < t} M_{t}$ is an isomorphism for all~$s \in \R$ \cite{Schmahl.2021}.
Alternatively, one may also impose the condition that $M$ be \emph{continuous from below at $t$} for all~$t \in \R$, i.e., $\colim_{s < t} M_{s} \to M_{t}$ is an isomorphism for all~$t \in \R$.
Continuity from above is satisfied in Morse's setting of studying compact sublevel set filtrations with \v{C}ech homology; in this setting, both the filtration as well as the resulting persistence module are continuous from above.
In terms of barcodes, continuity from above amounts to only admitting intervals that are closed on the left and open on the right.
One can also simply disregard the distinction between open, closed and half-open intervals in a barcode $\{I_{\lambda}\}_{\lambda \in \Lambda}$, which leads to another invariant that is more convenient in many settings, including our formulation of generalized Morse inequalities.
This invariant, known as the \emph{persistence diagram}, is defined as the multiset given by the \emph{multiplicity function} $\m \colon \mulDom \to \N$ that associates to an element in
\begin{equation*}
\mulDom =
\big\{ (p,q) \mid p \in \R \cup \{-\infty\}, \ q \in \R \cup \{+\infty\}, \ p < q \big\}
\end{equation*}
the cardinality of the set of intervals with lower bound $p$ and upper bound $q$, $\{ \lambda \in \Lambda \mid \inf I_{\lambda} = p,\ \sup I_{\lambda} = q\}$.
Since the type of interval is irrelevant in the observable category, the persistence diagram associated to the barcode of the radical is still a complete invariant for q-tame modules in the observable category.
Stated explicitly, we have the following.

\begin{thm}[\cite{Chazal.2016a, Chazal.2016b}] \label{t:q-tame modules have barcodes}
	Every q-tame persistence module has a unique persistence diagram that completely determines its isomorphism type in the observable category.
\end{thm}

We have thus seen how to obtain a persistence diagram from a real-valued function $f$ by applying a functor $\H \colon \Top \to \Vect$ to its sublevel set filtration and considering the persistence diagram of the resulting persistence module $\H(f_{\leq \bullet})$, which is well-defined provided that $\H(f_{\leq \bullet})$ is q-tame.
In this case, we will also call the function $f$ and the filtration $f_{\leq \bullet}$ \emph{q-tame} with respect to the functor~$\H$.
If $\H$ is a homology theory, we will call $\H(f_{\leq \bullet})$ the \emph{persistent homology} of the filtration~$f_{\leq \bullet}$.

As mentioned earlier, the passage from real-valued functions on a topological space to persistence diagrams is 1-Lipschitz for appropriate metrics.
We now provide some more details regarding this result for context, noting that they are not used in the present work.
On the space of real-valued functions one considers the metric induced by the supremum norm, and on the space of persistence diagrams the \emph{bottleneck distance}.
This metric expresses the distance between persistence diagrams by a matching of their points, counted with multiplicities, which is optimal with respect to the $L^{\infty}$-distance on $\mulDom$: two diagrams are within distance $\delta$ if any two matched points are within $L^{\infty}$-distance $\delta$, and any unmatched point has $L^{\infty}$-distance at most $\delta$ to the diagonal $\{(p,q) \in \R^2 \mid p = q\}$.

The most general stability result is shown by considering as an intermediate step the metric spaces of filtrations and persistence modules equipped with the \emph{interleaving distance}.
Given two functors $M,N \colon \R \to \mathsf C$, where $\mathsf C$ is typically $\Top$ or $\Vect$, a $\delta$-interleaving between $M$ and $N$ consists of a pair of natural transformations $(M_t \to N_{t+\delta})_t$ and $(N_t \to M_{t+\delta})_t$ from one diagram to a shifted version of the other and vice versa, which compose to the structure maps $(M_t \to M_{t+2\delta})_t$ and $(N_t \to N_{t+2\delta})_t$.
Clearly, the case $\delta=0$ describes an isomorphism, and the infimum of $\delta$ for which $M$ and $N$ admit a $\delta$-interleaving is defined as the interleaving distance between them.
$M$ and $N$ are isomorphic in the observable category if and only if they have interleaving distance $0$.
As it turns out, the stability of persistence barcodes can then be described as a sequence of 1-Lipschitz transformations: From functions (with the supremum norm) to filtrations by sublevel sets (with the interleaving distance), to persistence modules (with the interleaving distance), and finally to persistence diagrams (with the bottleneck distance).
In this approach, by far the most difficult step is showing that passing from persistence modules to persistence diagrams is stable, a result which is known as the Algebraic Stability Theorem \cite{Chazal.2009, Bauer.2015, Chazal.2016a}.

\section{Generalized Morse inequalities} \label{s:inequalities}

In this section we prove that $q$-tame persistence modules satisfy a general version of Morse inequalities, specializing to the usual Morse inequalities in the smooth context as well as to the version used by Morse and Tompkins to prove their Unstable Minimal Surface Theorem.
We deduce these general inequalities from the existence of persistence diagrams as reviewed in \cref{t:q-tame modules have barcodes}.

First recall that for a Morse function $f$ on a closed smooth manifold $X$, the classical Morse inequalities \cite{Morse.1925} state that for any non-negative integer~$n$ the following inequality holds:
\begin{equation} \label{e:classical morse inequalities}
\sum_{d=0}^n \ (-1)^{n-d} \big( c_{d}(f) - \beta_{d}(X) \big) \ \geq \ 0,
\end{equation}
where $c(f)$ is the number of critical points of $f$ with index $d$ and $\beta_{d}(X)$ is the $d\th$ Betti number of $X$.

If no two critical points of $f$ have the same value, the critical points are in natural one-to-one correspondence with the critical values, which, in turn, are in one-to-one correspondence with the homological changes in the sublevel set filtration of $f$, i.e., the endpoints of the intervals appearing in the barcode of the persistent homology of $f_{\leq \bullet}$.
More precisely, an index $d$ critical point either kills an existing homology class, in which case it corresponds to the right endpoint of an interval in the barcode of $H_{d-1}(f_{\leq \bullet})$, or it gives rise to a new homology class, in which case it corresponds to the left endpoint of an interval in the barcode of $H_d(f_{\leq \bullet})$.

The Betti numbers of $X$ may also be expressed in terms of barcodes, as they agree with the number of intervals that extend to $+\infty$.
Thus, the above Morse inequalities can be expressed entirely in terms of the barcode (or persistence diagram) of the persistent homology of the sublevel set filtration of the function, which encodes the homological changes in the filtration.

This approach of counting homological changes instead of critical points is also what Morse used in the non-smooth setting of functional topology.
To keep track of the number of $d$-dimensional homological events at filtration value $t$ that persist for at least time $\epsilon > 0$ but not indefinitely, Morse \cite{Morse.1940} defined the $(d, t, \epsilon)$-\emph{cap numbers} of a filtration.
The definition given by Morse is specific to Vietoris homology and implicitly relies on the fact that the resulting persistence module is continuous from above.
Expressed in terms of persistence barcodes, the $(d, t, \epsilon)$-cap number correspond to the number of bars in the $d\th$ barcode with left endpoint~$t$ and length greater than $\epsilon$, plus the number of bars in the $(d-1)\th$ barcode with right endpoint~$t$ and length greater than $\epsilon$.
In the compact smooth setting, for sufficiently small~$\epsilon$, the $(d, t, \epsilon)$-cap number equals the number of critical points of index~$d$ and value~$t$, which either create homology in degree $d$ or destroy homology in degree $(d-1)$.
In \cite[Corollary~12.3]{Morse.1940}, Morse proves a version of his eponymous inequalities using cap numbers as a replacement for numbers of critical points, with the stated goal of making the inequalities applicable in settings where the function may not be smooth or the number of critical points may not be finite.

We now take a more general persistence-based approach that allows us to go beyond the setting of Vietoris homology.
Working entirely in the algebraic setting, we will fix a graded q-tame persistence module $M$.
Of course, one may think of $M$ as the persistent homology of a q-tame filtration for any choice of homology theory, but $M$ could also arise, for example, as the filtered Floer homology of some Hamiltonian on a symplectic manifold.
By \cref{t:q-tame modules have barcodes}, $M$ has a persistence diagram in every degree $d$, with multiplicity function denoted $\m_d \colon \mulDom \to \N$.
In analogy to Morse's definitions, we may define, for an integer $d$ and real numbers $t$ and $\epsilon > 0$, the $(d, t, \epsilon)$-\emph{cap number} of our graded q-tame persistence module $M$ in terms of its persistence diagram as
\begin{equation*}
c_{d}^{\epsilon}(t) \ =
\alpha_{d}^{\epsilon}(t) + \omega_{d-1}^{\epsilon}(t),
\end{equation*}
where 
\begin{align*}
\alpha_{d}^{\epsilon}(t) &= \sum_{\substack{q \in \R \cup \{\infty\} \\ q - t > \epsilon}} \m_d(t, q),
&
\omega_{d}^{\epsilon}(t) &= \sum_{\substack{p \in \R \cup \{-\infty\} \\ t - p > \epsilon}} \m_{d}(p, t) 
\end{align*}
are the \emph{number of births} and the \emph{number of deaths}, respectively, in degree $d$, at parameter $t$, and with persistence greater than~$\epsilon$.
Note that finiteness of these quantities is ensured by the q-tameness of $M$ and the use of a non-zero $\epsilon$ bounding below the persistence of the considered features.
To see the necessity of this second condition, consider the q-tame persistence module given by the infinite product $\prod_{n \in \N_{> 0}} C([0,1/n))$ whose cap numbers $c^{\epsilon}(0)$ tend to $\infty$ as $\epsilon$ tends to 0.

Whenever the sums below are well-defined, we also consider the \emph{$(d,\epsilon)$-cap numbers}
\[
c_{d}^{\epsilon} = \sum_{t \in \R} c_{d}^{\epsilon}(t)  = \alpha_{d}^{\epsilon} + \omega_{d-1}^{\epsilon},
\]
where
\begin{align*}
\alpha_{d}^{\epsilon} &= \sum_{t \in \R} \alpha_{d}^{\epsilon}(t) = \sum_{\substack{(p,q) \in \mulDom \\ q - p > \epsilon \\ p \neq -\infty}}\m_d(p,q),
&
\omega_{d}^{\epsilon} &= \sum_{t \in \R} \omega_{d}^{\epsilon}(t) = \sum_{\substack{(p,q) \in \mulDom \\ q - p > \epsilon \\ q \neq \infty}} \m_{d}(p,q)
\end{align*}
are the \emph{total number of births} and the \emph{total number of deaths}, respectively, in degree~$d$ and with persistence greater than~$\epsilon$.

An important setting where all cap numbers are well-defined is when $M$ is the persistent homology of the sublevel set filtration of a bounded function.
A more general statement can be made using the following notion.

\begin{defi} \label{d:initially and eventually constant}
	A persistence module $M$ is said to be \emph{initially constant} if there is $s \in \R$ such that $M_{r,s}$ is an isomorphism for all $r \leq s$.
	Similarly, it is said to be \emph{eventually constant} if there is $u \in \R$
	such that $M_{u,v}$ is an isomorphism for all $u \leq v$.
\end{defi}

\begin{thm} \label{t:cap numbers well defined}
	Let $M$ be a q-tame persistence module that is both initially and eventually constant.
	If $\m$ is the multiplicity function of the persistence diagram of~$N$,
	then for each $\epsilon > 0$,
	\begin{equation*}
	\sum_{ \substack{ (p,q) \in \mulDom \\ q-p > \epsilon } } \m (p,q) < \infty.
	\end{equation*}
\end{thm}

\begin{proof}
	Let $s, u \in \R$ be as in \cref{d:initially and eventually constant}.
	We split the sum whose finiteness we want to show in two parts
	\begin{equation*}
	\sum_{ \substack{ (p,q) \in \mulDom \\ q-p > \epsilon } } \m (p,q) \ = \!\!
	\sum_{ \substack{ (p,q) \in \mulDom \setminus T \\ q-p > \epsilon } } \m (p,q) \ +
	\sum_{ \substack{ (p,q) \in T \\ q-p > \epsilon } } \m (p,q)
	\end{equation*}
	where $T$ denotes the triangle
	\begin{equation*}
	T = \{(p,q) \in \mulDom \mid s \leq p < q \leq u\}.
	\end{equation*}

	For the first summand, observe that since $M$ is constant below $s$ and constant above $u$, we have $\m(p,q) = 0$ whenever one of $-\infty < p < s$ or $q < s$ or $p > u$ or $u < q < \infty$ holds.
	This implies
	\begin{equation*}
	\sum_{ \substack{ (p,q) \in \mulDom \setminus T \\ q-p > \epsilon } } \m (p,q)
	\ = \!
	\sum_{s < q < u} \m (-\infty,q)
	\ + \!
	\sum_{s < p < u} \m (p, \infty)
	\end{equation*}
	which is clearly finite because $M$ is q-tame.

	For the second summand, note that
	\begin{equation*}
	\sum_{ \substack{ (p,q) \in T \\ q-p > \epsilon } } \m (p,q)
	\ \, \leq \!\!
	\sum_{(p,q) \in T^{\epsilon}} \m (p,q),
	\end{equation*}
	where $T^{\epsilon}$ is the smaller triangle
	\begin{equation*}
	T^{\epsilon} = \{(p,q) \in T \mid q-p \geq \epsilon\}.
	\end{equation*}
	Thus, in order to prove the theorem, it suffices to show that we have
	\begin{equation*}
	\sum_{(p,q) \in T^{\epsilon}} \m (p,q)
	\ < \
	\infty.
	\end{equation*}
	To do this, we consider open quadrants
	\begin{equation*}
	Q(x, y) = \{ (p, q) \in \R^2 \mid p < x \text{ and } y < q \}.
	\end{equation*}
	Covering the compact set $T^{\epsilon}$ by the open quadrants $Q \left(x, x + \frac{\epsilon}{2} \right)$ for $x \in \R$, we may choose a finite subcover given by, say, $x_1,\dots, x_n$.
	We obtain
	\begin{equation*}
	\sum_{(p,q) \in T^{\epsilon}} \m (p,q)
	\ \leq \
	\sum_{i=1}^n \sum_{\substack{(p, q) \in \\ Q (x_i, x_i + \frac{\epsilon}{2})}} \m(p,q).
	\end{equation*}
	Each of the sums $\sum_{(p,q) \in Q \left(x_i, x_i + \frac{\epsilon}{2} \right)} \m(p,q)$ over the quadrants $Q \left(x_i, x_i + \frac{\epsilon}{2} \right)$ is finite since $M$ is q-tame (which is where the name q-tame or \emph{quadrant}-tame comes from, see \cite[Section 3.8]{Chazal.2016a}).
\end{proof}

Comparing to the classical Morse inequalities, the cap numbers in dimension $d$ play the role of the number of critical points with index $d$.
As an analogue to the Betti numbers of the manifold appearing in the usual Morse inequalities, Morse defines quantities $p_{d}$, which we refer to as \emph{essential dimensions}.
For a graded q-tame persistence module $M$, these can be expressed in the language of persistence diagrams as
\[
p_{d} \ = \!\! \sum_{p \in \R \cup \{-\infty\}} \m_d(p,\infty),
\]
which agrees with the dimension of the colimit of the degree $d$ part of $M$.

\begin{thm} \label{t:inequalities}
	Let $\epsilon > 0$, and let $M$ be a graded q-tame persistence module
	with finite cap numbers $c_{d}^{\epsilon}$ and finite essential dimensions $p_{d}$ for all $d$.
	If $\m_d(-\infty, p) = 0$ for all $p \in \R \cup \{\infty\}$ and all $d$, then we have Morse inequalities
	\begin{equation} \label{e:morse inequalities}
	\sum_{d=0}^n \ (-1)^{n-d} (c_{d}^{\epsilon} - p_{d}) \ \geq\ 0
	\end{equation}
	for any dimension $n$.
\end{thm}

\begin{proof}
Recall that the $d$th $\epsilon$-cap number is defined as
\[c_{d}^{\epsilon} = \alpha_{d}^{\epsilon} + \omega_{d-1}^{\epsilon}.\] 
Since we assume $\m_d(-\infty, p) = 0$ for all $p$, we have
	\begin{equation*}
	p_{d} \ = \alpha_{d}^{\epsilon} - \omega_{d}^{\epsilon}.
	\end{equation*}
The difference of the two numbers is thus
\[
c_{d}^{\epsilon} - p_{d} = \omega_{d-1}^{\epsilon} + \omega_{d}^{\epsilon},
\]
and so their sum is
\[
\sum_{d=0}^n \ (-1)^{n-d} (c_{d}^{\epsilon} - p_{d}) = \omega_{n}^{\epsilon} \geq\ 0
\]
as claimed.
\end{proof}

In other words, given the assertion that the persistence module $M$ has a persistence diagram, the Morse inequalities simply express the fact that the total number of deaths of features with persistence greater than $\epsilon$ is nonnegative.
This observation illustrates the usefulness of interpreting fundamental facts in Morse theory through the lens of persistence theory.

As shown in \cref{t:cap numbers well defined}, the finiteness assumptions are satisfied if $M$ is initially and eventually constant.
Hence, as a special case, the theorem yields generalized Morse inequalities for any bounded real-valued function whose sublevel set filtration has q-tame persistent homology, including smooth Morse functions $f$ on closed smooth manifolds $X$.
As outlined in our motivation for the definition of cap numbers, in this setting our inequalities \eqref{e:morse inequalities} agree with the classical inequalities \eqref{e:classical morse inequalities}, as $\beta_d(X) = p_d$ and $c_d^{\epsilon} = \# \crit_d(f)$ for any $\epsilon > 0$ smaller than the minimum difference between any two critical values of~$f$.
Morse inequalities for unbounded functions can still be obtained by restricting the function to an arbitrary sublevel set.
Using \v{C}ech homology and considering bounded q-tame functions, our cap numbers and essential dimensions agree with the corresponding historical notions from \cite{Morse.1940}, so in this case our inequalities \eqref{e:morse inequalities} also agree with the inequalities \cite[Corollary~12.3]{Morse.1940}.

To apply our inequalities, one needs q-tameness, and our next goal will be to give topological conditions that ensure q-tameness (\cref{t:local connectedness implies q-tameness}).
These conditions are in particular satisfied by the Douglas functional (\cref{prop:douglas_hlc}), as shown by Morse and Tompkins in their work on unstable minimal surfaces \cite{Morse.1939}, which we will partly review in \cref{s:surfaces}, and which motivated the developments in \cite{Morse.1940}.

\begin{rem} \label{r:homotopically critial points}
	In addition to the formulation of the inequalities in terms of cap numbers, Morse also proposed a generalized version of critical points, which he called homotopically critical, and which formalizes the idea of criticality of a point in terms of topological changes in the sublevel set filtration.
	This notion was employed in the above mentioned work on minimal surfaces by Morse and Tompkins \cite{Morse.1939}.
	
	The usefulness of this notion might however be limited in some cases of interest.
	In Floer theory, for example, critical points of the action functional corresponding to a Hamiltonian usually do not have a finite index and thus do not lead to a change in the homotopy type of sublevel sets.
	In this setting, our approach of formulating the inequalities purely in algebraic terms might be more suitable: while these critical points are not topologically visible, they do correspond to features in the persistence diagram of the filtered Floer homology.
\end{rem}

\section{Existence of persistence diagrams} \label{s:connectivity}

We have seen that q-tame functions admit persistence diagrams, which can be used to formulate Morse inequalities.
With q-tameness being a rather abstract algebraic property, we now establish concrete topological conditions that ensure the q-tameness of a function.
Our definitions are motivated by similar conditions considered by Morse in his work on functional topology.
We will present a historical account in \cref{s:surfaces}.

Whether a function is q-tame or not depends on the functor that is used to pass from the sublevel set filtration of the function to a persistence module.
The general strategy is to deduce global finiteness properties like q-tameness from local ones.
Thus, the functors we consider should have a certain property allowing us to do so.

Specifically, let $\H = (\H_d)_{d \in \Z} \colon \Top \to \Vect$ be a fixed graded homotopy invariant functor, which we call a \emph{homology theory}.
A triple of spaces $X_{1}, X_{2} \subseteq X$ is said to have a \emph{Mayer--Vietoris sequence} for $\H$ if the inclusion-induced maps can be completed to a long exact sequence
\begin{equation*}
\begin{tikzcd}[column sep = small]
\cdots \arrow[r] &[-10pt] \H_{n+1}(X) \arrow[d] & &[-20pt] & \\
& \H_{n}(X_{1} \cap X_{2}) \arrow[r] &
\H_{n}(X_{1}) \oplus \H_{n}(X_{2}) \arrow[r] &
\H_{n}(X) \arrow[d] \\ & & &
\H_{n-1}(X_{1} \cap X_{2}) \arrow[r] &
\cdots \ .
\end{tikzcd}
\end{equation*}
We say that $\H$ has the \emph{open} (resp.\@\ \emph{compact}) \emph{Mayer--Vietoris property} if there are natural Mayer--Vietoris sequences for all triples $X_{1}, X_{2} \subseteq X$ with $X = X_1 \cup X_2$ and $X_i \subseteq X$ open (resp.\@\ compact Hausdorff).

For the rest of this section, we will assume that $\H$ is a homology theory that has either the open or the compact Mayer--Vietoris property and for which there is $n_0$ such that $\H_{n}$ is zero for all $n \leq n_0$.
We will also assume that $\H_{n}$ takes finite dimensional values on one-point spaces.
Note that this includes singular homology with field coefficients, which has the open Mayer--Vietoris property (like any homology theory in the sense of the Eilenberg--Steenrod axioms \cite[Section I]{Eilenberg.1952}), and it also includes \v{C}ech homology with field coefficients, which has the compact Mayer--Vietoris property.

\subsection{Local connectivity and q-tameness}

\begin{defi}
	For $n \in \Z$, a continuous map is said to be \emph{$n$-homologically small} or $\HS_n$ if the image of the map induced by $\H_{n}$ is finite dimensional.
	We omit references to $n$ if the condition holds for all integers.
\end{defi}

\begin{defi}
	The sublevel set filtration of a function $f \colon X \to \R$ is called \emph{locally homologically small} or \emph{$\HLC$} if for any $x \in X$, any neighborhood $V$ of $x$, and any pair of indices $s,t$ with $f(x) < s < t$ there is a neighborhood $U$ of $x$ with $U \subseteq V$ such that the inclusion $f_{\leq s} \cap U \hookrightarrow f_{\leq t} \cap V$ is $\HS$.
\end{defi}

\begin{defi}
	We say that a sublevel set filtration is \emph{compact} if all sublevel sets are compact Hausdorff spaces.
\end{defi}

If $f_{\leq t}$ is compact for all $t$, then the function $f$ is necessarily lower-semicontinuous and bounded from below (see \cite[p.~444]{Morse.1939} or \cite[Theorem 3.1]{Struwe.1988}).

Our main result, proven in the remainder of this subsection, is that for compact sublevel set filtrations the $\HLC$ condition implies \mbox{q-tameness}, and consequently also the existence of a persistence diagram:

\begin{thm} \label{t:local connectedness implies q-tameness}
	If the sublevel set filtration of a function $f \colon X \to \R$ is compact and	$\HLC$, then its persistent homology is also q-tame.
	In particular, $f$ has a persistence diagram.
\end{thm}

The general proof strategy is inspired by the proof of Wilder's Finiteness Theorem \cite[p.~325]{Wilder.1949} as presented by Bredon \cite[Section II.17]{Bredon.1997}.
We collect the main ideas in several lemmas.

\begin{lem} \label{l:commutative algebra}
	Given a commutative diagram of modules over a principal ideal domain
	\begin{equation*}
	\begin{tikzcd}
	A_{1,1} \arrow[r] & A_{1,2} & \\
	A_{2,1} \arrow[r] \arrow[u] & A_{2,2} \arrow[r] \arrow[u] & A_{2,3} \\
	& A_{3,2} \arrow[r] \arrow[u] & A_{3,3} \arrow[u]
	\end{tikzcd}
	\end{equation*}
	where the middle row is exact and both $A_{2,1} \to A_{1,1}$ and $A_{3,3} \to A_{2,3}$ have finitely generated images, then so does $A_{3,2} \to A_{1,2}$.
\end{lem}

\begin{proof}
	This is proven via a straightforward diagram chase.
	For more details see \cite[Lemma II.17.3]{Bredon.1997}.
\end{proof}

\begin{lem} \label{l:neighborhood third}
	Let $X$ be a locally compact Hausdorff space.
	For any compact subset~$K$ and open set $U$ with $K \subseteq U$ there exists a compact set $K^\prime$ such that
	\begin{equation*}
	K \subseteq \interior(K^\prime) \subseteq K^\prime \subseteq U.
	\end{equation*}
\end{lem}

\begin{proof}
	For any $x \in K$ choose a compact neighborhood $C(x) \subseteq U$.
	We have
	\begin{equation*}
	K \subseteq \bigcup_{x \in K} \interior(C(x)).
	\end{equation*}
	Since $K$ is compact, there is a finite subset $\{x_1, \dots, x_m\}$ of elements in $K$ so that
	\begin{equation*}
	K \subseteq \bigcup_{i=1}^m \interior(C(x_i)) \subseteq \interior\left(\bigcup_{i=1}^m C(x_i)\right) \subseteq \bigcup_{i=1}^m C(x_i) \subseteq U
	\end{equation*}
	Defining $K^\prime = \bigcup_{i=1}^m C(x_i)$ finishes the proof.
\end{proof}

We want to use \cref{l:neighborhood third} on the domain of the function whose sublevel set filtration we consider.
However, we do not want to assume the domain to be locally compact for \cref{t:local connectedness implies q-tameness}.
To circumvent this, we will work in one of the sublevel sets, which are assumed to be compact Hausdorff and hence locally compact.
This requires the use of a slight weakening of the $\HLC$ condition.

\begin{defi}
	For $u \in \R$, the sublevel set filtration of a function $f \colon X \to \R$ is called \emph{$\HLC$ below $u$} if, for any $x \in X$, any neighborhood $V$ of $x$, and any pair of indices $s,t$ with $f(x) < s < t < u$, there is a neighborhood $U$ of $x$ with $U \subseteq V$ such that the inclusion $f_{\leq s} \cap U \hookrightarrow f_{\leq t} \cap V$ is $\HS$.
\end{defi}

\begin{lem} \label{l:restriction of LHS filtration}
	Let $f \colon X \to \R$ be a function whose sublevel set filtration is $\HLC$.
	Fix $u \in \R$ and let $g \colon Y \to \R$ be the restriction of $f$ to the sublevel set $Y = f_{\leq u}$.
	Then the sublevel set filtration defined by $g$ is $\HLC$ below $u$.
\end{lem}

\begin{proof}
	Let $x \in Y$, let $V$ be a neighborhood of $x$ in $Y$, and consider indices $s,t$ with $g(x) < s < t < u$.
	We need to find a neighborhood $U \subseteq V$ of $x$ such that the inclusion $g_{\leq s} \cap U \hookrightarrow g_{\leq t} \cap V$ is $\HS$.

	Since $Y \subseteq X$ carries the subspace topology, we may choose a neighborhood $V'$ of $x$ in $X$ such that $V = V' \cap Y$.
	The sublevel set filtration of $f$ is assumed to be $\HLC$, so there is a neighborhood $U' \subseteq V'$ of $x$ in $X$ such that the inclusion $f_{\leq s} \cap U' \hookrightarrow f_{\leq t} \cap V'$ is $\HS$.
	We set $U = U' \cap Y$, which defines a neighborhood of $x$ in $Y$.

	Now $s < t < u$ implies that $g_{\leq s} = f_{\leq s}$ and $g_{\leq t} = f_{\leq t}$.
	Moreover, we have $f_{\leq s} \cap Y = f_{\leq s} \cap f_{\leq u} = f_{\leq s}$ and $f_{\leq t} \cap Y = f_{\leq t} \cap f_{\leq u} = f_{\leq t}$.
	Thus, we obtain $g_{\leq s} \cap U = f_{\leq s} \cap Y \cap U' = f_{\leq s} \cap U'$ and $g_{\leq t} \cap V = f_{\leq t} \cap Y \cap V' = f_{\leq t} \cap V'$.
	This implies that the inclusion $g_{\leq s} \cap U \hookrightarrow g_{\leq t} \cap V$ is $\HS$ because it agrees with the inclusion $f_{\leq s} \cap U' \hookrightarrow f_{\leq t} \cap V'$, which is $\HS$ by assumption.
	This finishes the proof.
\end{proof}

\begin{lem} \label{l:key lemma for q-tameness}
	Let $f \colon Y \to \R$ be a function on a locally compact Hausdorff space~$Y$ whose sublevel set filtration is compact and $\HLC$ below $u \in \R$, and consider subsets $C \subseteq L \subseteq Y$ with $C$ compact and $L$ open.
	For any $s < t < u$ the inclusion $C \cap f_{\leq s} \hookrightarrow L \cap f_{\leq t}$ is $\HS$.
\end{lem}

\begin{proof}
	Recall our assumption that the underlying homology theory $\H$ has either the open or the compact Mayer--Vietoris property and that there is some $n_0$ such that $\H_{n}$ is zero for all $n \leq n_0$.
	The statement of the lemma holds for $\HS_{n}$ in place of $\HS$ for any $n \leq n_0$ since $\H_{n}$ induces the zero map.
	We will proceed by induction on $n \geq n_0$ assuming the statement for $\HS_{n-1}$.

	We define $\Sigma_{s, t}$ to be the collection of all open subsets $V \subseteq Y$ whose closure $\overline{V}$ is compact, contained in $L$, and has an open neighborhood $U$ with $\overline{V} \subseteq U \subseteq L$	for which there exists $s' \in (s,\, t)$ such that the inclusion $U \cap f_{\leq s'} \hookrightarrow L \cap f_{\leq t}$ is $\HS_n$.
	We will show that $\Sigma_{s, t}$ has the following three properties:
	\begin{enumerate}
		\item Any point $x \in L \cap f_{\leq s}$ has a neighborhood $V_x \in \Sigma_{s,t}$.
		\item If $V_1,\, V_2 \in \Sigma_{s,t}$ then $V_1 \cup V_2 \in \Sigma_{s,t}$.
		\item For each $V \in \Sigma_{s,t}$ the inclusion
		$V \cap f_{\leq s} \hookrightarrow L \cap f_{\leq t}$ is $\HS_n$.
	\end{enumerate}

	Assuming them for the moment, the first property allows us to cover $C \cap f_{\leq s}$ by sets $V_x \in \Sigma_{s,t}$, $x \in C \cap f_{\leq s}$.
	Because both $C$ and $f_{\leq s}$ are compact, $C \cap f_{\leq s}$ is again compact, and hence the cover can be chosen finite, represented by say $x_1,\dots, x_m$.
	By the second property, we have
	\[V \defeq \bigcup_{i = 1}^m V_{x_i} \in \Sigma_{s,t}.\]
	Using the third property, the inclusion
	$V \cap f_{\leq s} \hookrightarrow L \cap f_{\leq t}$
	is $\HS_n$.
	The inclusion
	$C \cap f_{\leq s} \hookrightarrow L \cap f_{\leq t}$
	factors through the previous one, so it is $\HS_n$ as well.
	What is left to do is to show that $\Sigma_{s,t}$ has the three claimed properties.

	For the third property, let $V \in \Sigma_{s,t}$ with $U$ an open neighborhood of $\overline{V}$ in $L$ and $s' \in (s,\, t)$ such that the inclusion
	$U \cap f_{\leq s'} \hookrightarrow L \cap f_{\leq t}$
	is $\HS_n$.
	Again, the inclusion
	$V \cap f_{\leq s} \hookrightarrow L \cap f_{\leq t}$
	factors through the previous one, so it is $\HS_n$ as well.
	Thus, $\Sigma_{s, t}$ has the third property we want.

	Next, we will show using the $\HLC$ property that $\Sigma_{s, t}$ has the first required property, i.e., that any point $x \in L \cap f_{\leq s}$ has a neighborhood in $\Sigma_{s, t}$.
	Choose an arbitrary $s' \in (s,\, t)$.
	Since the sublevel set filtration of $f$ is $\HLC$ below $u$ and we have $f(x) \leq s < s' < t < u$, there is an open neighborhood $U_x \subseteq L$ such that the inclusion
	$U_x \cap f_{\leq s'} \hookrightarrow L \cap f_{\leq t}$
	is $\HS$, so in particular $\HS_n$.
	By local compactness of $Y$ we can choose a compact neighborhood $K_x$ of $x$ contained in $U_x$.
	Now $V_x = \interior (K_x)$ is a neighborhood of $x$ with $V_x \in \Sigma_{s,t}$.

	Finally, using the Mayer--Vietoris property and the induction hypothesis we will show that $\Sigma_{s,t}$ has the second required property, i.e., that it is closed under finite unions.
	So for $i \in \{1, 2\}$ let $V_i \in \Sigma_{s,t}$ with $U_i$ and $s'_i \in (s,\, t)$ such that
	$\overline{V_i} \subseteq U_i \subseteq L$
	and
	$U_{i} \cap f_{\leq s'_i} \hookrightarrow L \cap f_{\leq t}$
	is $\HS_n$.
	Writing $K_i = \overline{V_i}$, we use \cref{l:neighborhood third} to construct compact sets $K'_i$ such that
	\begin{equation*}
	V_i \subseteq K_i \subseteq V'_i \subseteq K'_i \subseteq U_i \subseteq L
	\end{equation*}
	where $V'_i = \interior(K'_i)$.
	The union $V_1 \cup V_2 \subseteq L$ is open, its closure $\overline{V_1 \cup V_2}$ is compact, and $\overline{V_1 \cup V_2} \subseteq V'_1 \cup V'_2 \subseteq L$.
	Thus, we obtain $V_1 \cup V_2 \in \Sigma_{s,t}$ if we can show that there is an $s' \in (s,\, t)$ such that the inclusion
	$\left(V'_1 \cup V'_2 \right) \cap f_{\leq s'} \hookrightarrow L \cap f_{\leq t}$
	is $\HS_n$.
	To do so, we set $s'' = \min_i s'_i$ and choose $s' \in (s,\, s'')$.
	For proving that $\left(V'_1 \cup V'_2 \right) \cap f_{\leq s'} \hookrightarrow L \cap f_{\leq t}$ is $\HS_n$ we now distinguish the two cases where $\H$ has either the open or the compact Mayer--Vietoris property.

	For the open Mayer--Vietoris property, note that for both $i \in \{1, 2\}$ the inclusions
	$U_i \cap f_{\leq s''} \hookrightarrow L \cap f_{\leq t}$
	are $\HS_n$.
	Moreover, the inclusion
	$V'_1 \cap V'_2 \cap f_{\leq s'} \hookrightarrow U_1 \cap U_2 \cap f_{\leq s''}$
	is $\HS_{n-1}$ because it factors through the inclusion
	$K'_1 \cap K'_2 \cap f_{\leq s'} \hookrightarrow U_1 \cap U_2 \cap f_{\leq s''}$,
	which is $\HS_{n-1}$ by the induction hypothesis.
	Because the $V_i$ and $V'_i$ are open and because $\H$ has the open Mayer--Vietoris property, we obtain the following commutative diagram satisfying the assumptions of \cref{l:commutative algebra}:
	\[
	\begin{tikzcd}[column sep=small,every matrix/.append style={nodes={font=\small}}]
	\H_n(L \cap f_{\leq t}) \oplus \H_n(L \cap f_{\leq t}) \arrow[r] &
	\H_n(L \cap f_{\leq t}) & \\
	\H_{n}(U_1 \cap f_{\leq s''}) \oplus \H_n(U_2 \cap f_{\leq s''}) \arrow[r] \arrow[u] &
	\H_{n}((U_1 \cup U_2) \cap f_{\leq s''}) \arrow[r] \arrow[u] &
	\H_{n-1}(U_1 \cap U_2 \cap f_{\leq s''}) \\ &
	\H_{n}((V'_1 \cup V'_2) \cap f_{\leq s'}) \arrow[r] \arrow[u] &
	\H_{n-1}(V'_1 \cap V'_2 \cap f_{\leq s'}) \arrow[u].
	\end{tikzcd}
	\]
	We conclude that the inclusion
	$\left(V'_1 \cup V'_2 \right) \cap f_{\leq s'} \hookrightarrow L \cap f_{\leq t}$
	is $\HS_n$, which finishes this part of the proof.

	For the compact Mayer--Vietoris property, we apply \cref{l:neighborhood third} once more to obtain compact sets $K''_i$ such that
	\begin{equation*}
	V_i \subseteq K_i \subseteq V'_i \subseteq K'_i \subseteq V''_i \subseteq K''_i \subseteq U_i \subseteq L
	\end{equation*}
	where $V''_i = \interior(K''_i)$.
	The rest of the proof is then analogous to the previous case:
	We have that for both $i \in \{1, 2\}$ the inclusion
	$K''_i \cap f_{\leq s''} \hookrightarrow L \cap f_{\leq t}$
	is $\HS_n$ because it factors through $U_i \cap f_{\leq s''} \hookrightarrow L \cap f_{\leq t}$.
	Moreover, the inclusion
	$K'_1 \cap K'_2 \cap f_{\leq s'} \hookrightarrow K''_1 \cap K''_2 \cap f_{\leq s''}$
	is $\HS_{n-1}$ because it factors through the inclusion
	$K'_1 \cap K'_2 \cap f_{\leq s'} \hookrightarrow V''_1 \cap V''_2 \cap f_{\leq s''}$,
	which is $\HS_{n-1}$ by the induction hypothesis.
	Because the $K'_i$ and $K''_i$ as well as the sublevel sets of $f$ are all compact and because $\H$ has the compact Mayer--Vietoris property, we obtain the following commutative diagram satisfying the assumptions of \cref{l:commutative algebra}:
	\[
	\begin{tikzcd}[column sep=tiny,every matrix/.append style={nodes={font=\small}}]
	\H_n(L \cap f_{\leq t}) \oplus \H_n(L \cap f_{\leq t}) \arrow[r] &
	\H_n(L \cap f_{\leq t}) & \\
	\H_{n}(K''_1 \cap f_{\leq s''}) \oplus \H_n(K''_2 \cap f_{\leq s''}) \arrow[r] \arrow[u] &
	\H_{n}((K''_1 \cup K''_2) \cap f_{\leq s''}) \arrow[r] \arrow[u] &
	\H_{n-1}(K''_1 \cap K''_2 \cap f_{\leq s''}) \\ &
	\H_{n}((K'_1 \cup K'_2) \cap f_{\leq s'}) \arrow[r] \arrow[u] &
	\H_{n-1}(K'_1 \cap K'_2 \cap f_{\leq s'}) \arrow[u].
	\end{tikzcd}
	\]
	We conclude that the inclusion
	$\left(K'_1 \cup K'_2 \right) \cap f_{\leq s'} \hookrightarrow L \cap f_{\leq t}$
	is $\HS_n$, and so the same is true for the inclusion
	$\left(V'_1 \cup V'_2 \right) \cap f_{\leq s'} \hookrightarrow L \cap f_{\leq t}$
	as it factors through the previous one.
\end{proof}

We can now complete the proof of the claim stating that for compact sublevel set filtrations, $\HLC$ implies q-tameness.

\begin{proof}[Proof of \cref{t:local connectedness implies q-tameness}]
	By definition, the sublevel set filtration of $f$ is q-tame if and only if the inclusion $f_{\leq s} \hookrightarrow f_{\leq t}$ is $\HS$ for all pairs $s < t$.
	Choose $u \in \R$ with $u > t$ and let $g \colon Y \to \R$ be the restriction of $f$ to the sublevel set $Y=f_{\leq u}$.
	Since we assume $f$ to induce a $\HLC$ sublevel set filtration, by \cref{l:restriction of LHS filtration} the sublevel set filtration of $g$ is $\HLC$ below $u$.
	Clearly, the sublevel set filtration of $g$ is also compact, and its domain $Y$ is locally compact being a compact Hausdorff space by assumption.
	Thus, we can apply \cref{l:key lemma for q-tameness} to the filtration $g_{\leq \bullet}$ with $C = L = Y$ to obtain that the inclusion $f_{\leq s} = C \cap g_{\leq s} \hookrightarrow L \cap g_{\leq t} = f_{\leq t}$ is $\HS$.
\end{proof}

\subsection{Continuity}

We now describe a weaker version of local connectivity of sublevel set filtrations, which implies q-tameness when continuity is assumed.

\begin{defi}
	The sublevel set filtration of a function $f \colon X \to \R$ is said to be \emph{weakly locally homologically small} or \emph{weakly $\HLC$} if for any $x \in X$, any neighborhood $V$ of $x$, and any index $t > f(x)$, there is an index $s$ with $f(x) < s < t$ and a neighborhood $U$ of $x$ with $U \subseteq V$ such that the inclusion $f_{\leq s} \cap U \hookrightarrow f_{\leq t} \cap V$ is $\HS$.
\end{defi}

Clearly, any $\HLC$ sublevel set filtration is also weakly $\HLC$:
while the weak $\HLC$ property merely requires the existence of an index $s \in (f(x),t)$ satisfying the $\HS$ condition, the $\HLC$ property requires the $\HS$ condition to hold for any $s \in (f(x),t)$.
If the filtration is induced by a continuous function, the converse also holds, as the following theorem shows.

\begin{lem} \label{l:weak hlc to hlc}
	If the sublevel set filtration of a continuous function $f \colon X \to \R$ is weakly $\HLC$, then it is also $\HLC$.
\end{lem}

\begin{proof}
	Fix $x \in X$, a neighborhood $V$ of $x$, and indices $f(x) < s < t$.
	We need to show that there exists a neighborhood $U \subseteq V$ of $x$ such that the inclusion $f_{\leq s} \cap U \hookrightarrow f_{\leq t} \cap V$ is $\HS$.

	To do so, we start by using the weak $\HLC$ property to choose a neighborhood $U' \subseteq V$ of $x$ and an index $s' \in (f(x),\, t)$ such that the inclusion $f_{\leq s'} \cap U' \hookrightarrow f_{\leq t} \cap V$ is $\HS$.
	Now, we choose $U = f_{< s'} \cap U'$, where $f_{< s'} = f^{-1} (-\infty, s')$.
	Note that this choice of $U$ still defines a neighborhood of $x$ because $f$ is assumed to be continuous, so that $f_{< s'}$ is an open subset of~$X$.

	We obtain that $f_{\leq s} \cap U \subseteq f_{\leq s'} \cap U'$, so that the inclusion $f_{\leq s} \cap U \hookrightarrow f_{\leq t} \cap V$ is $\HS$  as it factors through the inclusion $f_{\leq s'} \cap U' \hookrightarrow f_{\leq t} \cap V$, which is $\HS$.
\end{proof}

The following result is deduced directly from \cref{l:weak hlc to hlc} and \cref{t:local connectedness implies q-tameness}.
The existence of a result of this kind has been suggested by Weinberger \cite{Weinberger.2011}, and a multiparameter version with slightly stronger assumptions on the domain of the function has been shown by Cagliari and Landi \cite{Cagliari.2011} .

\begin{cor} \label{c:q-tameness for continuous functions}
	If the sublevel set filtration of a continuous function $f \colon X \to \R$ is compact and weakly $\HLC$, then it is also q-tame.
\end{cor}

As the following example illustrates, the continuity assumption in the above corollary is crucial.

\begin{figure}[t]
	\centering
	\begin{tikzpicture}[scale = 60]
		\draw[thick,lightgray] (-.1,-.001) rectangle (.1,.05);
		\clip (-.1,0) rectangle (.1,.05);
		\foreach \i in {1,...,100}{
			\draw[line width=0.4/\i^0.25 pt] (0, 1/\i^2) circle (1/\i^2);
		}
	\end{tikzpicture}
	\caption{A closeup of the Hawaiian earring $\mathbb{H}^1$.}
\end{figure}
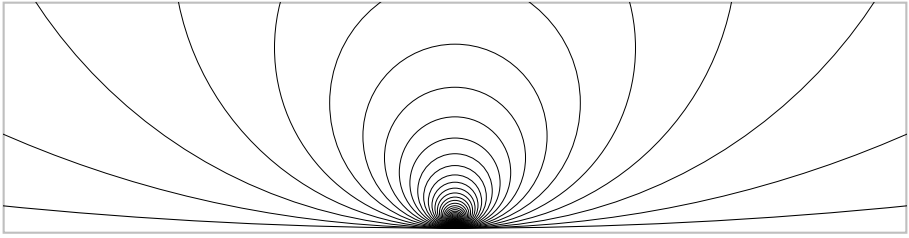

\begin{example} \label{e:counterexample}
	Consider the \emph{$d$-dimensional Hawaiian earring}
	\begin{equation*}
		\HE = \bigcup_{n \in \N} \left\{ (x_0, \dots, x_d) \in \R^{d+1} \ \middle | \ \left( x_0 - \frac{1}{n} \right)^2 \!\! + x_1^2 + \dots + x_d^2 = \left( \frac{1}{n} \right)^2 \right\},
	\end{equation*}
	which is a compact subspace of $\R^{d+1}$.
	The function $f \colon \HE \to \R$ whose value is $0$ at the origin and is $1$ everywhere else defines a compact and weakly $\HLC$ sublevel set filtration that is not q-tame with respect to any homology theory $\H$ for which $\H_{d}(\HE)$ is infinite dimensional, as is the case for both singular and \v Cech homology.
	Specifically, using the fact that \v{C}ech homology of compact Hausdorff spaces commutes with inverse limits, it is straightforward to verify that the \v{C}ech homology in degree $d$ of the $d$-dimensional Hawaiian earring $\HE$ is isomorphic to $\prod_{n\in\N}\F$, which is infinite dimensional over $\F$.
	Moreover, the singular homology of $\HE$ is also infinite dimensional, as proven in \cite{Barratt.1962}.

	To verify that $f$ has compact sublevel sets we notice that all sublevel sets are either the empty set, the singleton containing the origin, or $\HE$ itself, all compact Hausdorff spaces.

	In order to verify that the sublevel set filtration of $f$ is weakly $\HLC$, we consider $x \in \HE$, $V$ a neighborhood of $x$ in $\HE$, and $t > f(x)$.
	We need to find a neighborhood $U \subseteq V$ of $x$ and $s \in (f(x), t)$ such that the inclusion $f_{\leq s} \cap U \hookrightarrow f_{\leq t} \cap V$ is $\HS$.
	Since we assume that $\H_{n}$ is homotopy invariant and takes singletons to finite dimensional spaces, it suffices to find $U$ as above such that $f_{\leq s} \cap U \hookrightarrow f_{\leq t} \cap V$ is homotopic to a constant map.

	If $x$ is the origin, we have $f(x) = 0$ and choose $s \in (0, \min\{t, 1\})$.
	Then $f_{\leq s} = \{x\}$, so with $U = V$ the inclusion $f_{\leq s} \cap U \hookrightarrow f_{\leq t} \cap V$ is the inclusion of $\{x\}$ into $f_{\leq t} \cap V$, which is a constant map, so the weak $\HLC$ condition is trivially satisfied.

	For $x$ different from the origin we have $f(x) = 1$ and choose $s \in (1,t)$ arbitrarily, so that $f_{\leq s} = f_{\leq t} = \HE$.
	Note that since $x$ is not the origin, there is a unique $d$-sphere in $\HE$ that contains $x$.
	Clearly, we may choose $\delta > 0$ so small that $B_{\delta}(x) = \{y \in \R^{d+1} \mid \Vert x - y \Vert < \delta\} \cap \HE$ is a topological ball contained in this sphere and contained in $V$.
	The ball $B_\delta(x)$ can be contracted to $\{x\}$ in $V$, so choosing $U = B_{\delta}(x)$, we obtain that the inclusion $f_{\leq s} \cap U \hookrightarrow f_{\leq t} \cap V$ is homotopic to the constant map with value $x$.

	It remains to be shown that $f_{\leq \bullet}$ is not q-tame for $\H$.
	This follows directly from our assumption that $\H_{d}(\HE)$ is not finite dimensional, as $f_{\leq t}$ is constant with value $\HE$ for $t \geq 1$.
\end{example}
	
\section{Persistent homology and functional topology} \label{s:surfaces}

Having established topological conditions for the existence of persistence diagrams associated to filtrations and corresponding generalized Morse inequalities, we now describe how these results relate to Morse's general theory of functional topology \cite{Morse.1937, Morse.1938, Morse.1940}, and to its application in Morse and Tompkins' work on unstable minimal surfaces from \cite{Morse.1939}.
We will focus on the homological aspects of this approach, referring to \cite[Sections 4.3--5]{Bott.1980} for a general exposition.
For a more thorough presentation of the unstable minimal surface problem, including the analytical details see \cite[Section II.6]{Struwe.1988}, and for a historical account and an overview of subsequent results see \cite[Section 6]{Dierkes.2010} and \cite[Section 6.8.1]{Dierkes.2010b}.

\subsection{The unstable minimal surface problem}
\label{subsec:unstable}

Morse and Tompkins considered the following setting introduced by Douglas.
Let $g \colon \R \to \R^n$ be a $2\pi$-periodic function representing a simple closed curve such that $g$ is differentiable with Lipschitz derivative.
Let $\widetilde{\Omega}$ be the space of continuous non-decreasing functions $\varphi \colon \R \to \R$ with $\varphi(t+2\pi) = \varphi(t) + 2\pi$ for all $t$ and $\varphi(\alpha_i)=\alpha_i$ for three fixed distinct points $\alpha_i \in [0,2\pi)$.
The \emph{Douglas functional} on $\widetilde \Omega$ associated to the curve $g$ is defined as
\begin{equation*}
A_g(\varphi) = \frac{1}{16 \pi} \int_0^{2\pi} \int_0^{2\pi} \left\| \frac{g(\varphi(\alpha)) - g(\varphi(\beta))}{\sin \frac{\alpha-\beta}{2}} \right\|_2^2 \ \mathrm{d}\alpha \ \mathrm{d}\beta.
\end{equation*}
It coincides with the Dirichlet energy of the unique harmonic extension of the reparametrized curve $g \circ \varphi$ to a parametrized surface.
The Dirichlet energy is an upper bound for the area, with equality if the parametrization is conformal.
Let $\Omega_g = \{\varphi \in \widetilde\Omega \mid A_g(\varphi) < \infty\}$, equipped with the $C^0$ metric.
The set $\Omega_g$ is non-empty and the sublevel sets of $A_g$ are compact \cite[p.~448]{Morse.1939}.
Since $A_g$ is bounded below by $0$, this implies that $A_g$ attains a global minimum.
The corresponding surface is then a solution of \emph{Plateau's Problem}, which asks for a surface homeomorphic to a disk with boundary $g$ and minimum area.

As alluded to in \cref{r:homotopically critial points}, Morse and Tompkins consider a homotopical notion of critical point for a general function $F \colon M \to \R$ on a metric space \cite[p.~445]{Morse.1939}, see also \cite{Morse.1943}.

\begin{defi}[{{\cite[Definition II.6.1-II.6.2]{Struwe.1988}, \cite[p.~445, 472]{Morse.1939}}}]
\label{def:critical_set}
	Consider a real-valued function $F$ on a metric space $(M,d)$.
	A point $p \in M$ is called \emph{homotopically regular} if there exists a neighborhood $U$ of $p$ in $F_{\leq F(p)}$ and a continuous map $\varphi \colon U \times [0,1] \to M$ satisfying $\varphi(p,1) \neq p$ such that for every compact subset $V \subseteq U$ there exists a function $\delta \colon \R_{\geq 0} \to \R_{\geq 0}$ with
	$\delta(e) = 0$ if and only if $e = 0$, and
	\[
	F(\varphi(x,s)) - F(\varphi(x,t)) > \delta(d(\varphi(x,s),\varphi(x,t)))
	\]
	for all $x \in V$ and $0 \leq s \leq t \leq 1$.
	A point that is not homotopically regular is called \emph{homotopically critical}.
	Function values of homotopically critical points will be called \emph{critical values} and all other values will be called \emph{regular values}.
	A \emph{critical set} $S$ is a closed and open subset of the subspace of all homotopically critical points with a given function value.
	It is said to be of \emph{minimum type} if there exists a neighborhood $N$ of the closure $\overline{S}$ of $S$, taken in $M$, such that the function values on $N \setminus S$ strictly exceed the function value on $S$.
\end{defi}

Note that, in particular, an isolated local minimum constitutes a critical set of minimum type. Similarly, a critical submanifold of a Morse-Bott function on which the function values are locally minimized is also a critical set of minimum type.

Morse and Tompkins state the following Mountain Pass Theorem under slightly different assumptions.
It is worth noting, however, that the details in \cite{Morse.1939} are incomplete, with some crucial theorems such as \cite[Theorems 7.3 and 7.4, Corollary 7.1]{Morse.1939} being stated without proof, and with a citation to a paper in preparation that has never been published under the given name (we suppose that this paper is \cite{Morse.1940}).
Moreover, there is a gap in \cite{Morse.1940}, because \cite[Theorem 6.3]{Morse.1940}, which establishes q-tameness, is incorrect as we will show in \cref{c:counterexample}.
The assumptions we choose for our version of the Mountain Pass Theorem are adapted from the original assumptions to the modern language of persistence theory, and they fix the problem with q-tameness.
Still, our assumptions can be established for the Douglas functional from the results of Morse and Tompkins \cite{Morse.1939}.
We will comment in more detail on the differences between the assumptions in \cref{rem:critical_set,rem:mountain_pass_assumptions,subsec:historic_hlc}.

\begin{thm}[Mountain Pass Theorem, {\cite[Corollary 7.1]{Morse.1939}}]
\label{thm:mountain_pass}
	Let $M$ be a connected metric space and $F \colon M \to \R$ a function.
	Assume that $F$ is bounded below and that the sublevel set filtration of $F$ is compact and	$\HLC$ with respect to \v Cech homology.
	If $M$ contains two distinct critical sets of $F$ of minimum type, then it also contains a critical set not of minimum type.
\end{thm}

\begin{rem}\label{rem:critical_set}
	Our \cref{def:critical_set} of a critical set $S$ of \emph{minimum type} differs slightly from that of Morse and Tompkins \cite[p.~472]{Morse.1939}, who do not require the neighborhood $N$ of $S$ on which the function values exceed those on $S$ to contain the closure of $S$.
	Without this additional assumption, however, \cref{thm:mountain_pass} does not hold, as shown by the example $f \colon [0,1] \to \R$ with $f(0) = f(1) = 0$ and $f(t) = 1$ for $0 < t < 1$:
	$f$ has the four critical sets $\{0\}$, $\{1\}$, $\{0\} \cup \{1\}$ and $(0,1)$, which all satisfy the minimum type condition if the neighborhood $N$ need not contain their closure, but then there is no critical set that is not of minimum type.
\end{rem}

\begin{rem}\label{rem:mountain_pass_assumptions}
    In \cite[Corollary 7.1]{Morse.1939}, the assumptions that Morse and Tompkins use for $F$ and $M$ are that the sublevel set filtration is compact, ``regular at infinity'',``weakly upper-reducible'', and that $M$ is ``locally $F$-connected''.
    Compactness is also required for our version, regularity at infinity roughly corresponds to our assumption that $M$ is connected, and local $F$-connectedness is replaced by the $\HLC$ condition.
    (This last point will be discussed in more detail in \cref{subsec:historic_hlc}).
    
    The weak upper-reducibility condition does not have an analogue in the assumptions of \cref{thm:mountain_pass}. 
    This is interesting insofar as establishing that the Douglas functional satisfies this condition takes up a large part of \cite{Morse.1939}, and Courant \cite{Courant.1941} emphasizes that establishing weak upper-reducibility is one of the key difficulties in the proof of Morse and Tompkins, requiring ``a rather deep explicit analysis of the functional''.
    
    One can check that the example given in \cref{rem:critical_set} also satisfies the conditions that Morse and Tompkins impose for their Mountain Pass Theorem.

\end{rem}

\begin{rem}
	We mention that more general homology theories can be considered in \cref{thm:mountain_pass}.
	The precise hypotheses on the homology theory $H$ are that it is additive, taking non-zero values on non-empty sets in dimension $0$, and such that $\H(F_{\leq \bullet})$ is continuous from above, i.e., $\H(F_{\leq s}) \to \lim_{s < t} \H(F_{\leq t})$ is an isomorphism for all~$s \in \R$.
	One may also replace the topological conditions of $M$ being connected and $F$ being $\HLC$ by the algebraic conditions that $p_0 = 1$ and $\H(F_{\leq \bullet})$ be q-tame.
\end{rem}

Morse and Tompkins show \cite[Theorem 6.2]{Morse.1939} that each homotopically critical point of the Douglas functional $A_g$ indeed corresponds to
a \emph{minimal surface} -- a surface with vanishing mean curvature -- and using this correspondence the following result, also reviewed in \cite[Theorem II.6.10]{Struwe.1988}, can be deduced from the Mountain Pass Theorem.

\begin{thm}[{Unstable Minimal Surface Theorem \cite[p.~472]{Morse.1939}}]
\label{thm:unstable_minimial_surface}
	If the space~$\Omega_g$ contains two minimal surfaces contained in distinct critical sets of minimum type of the functional $A_g$, then it also contains an \emph{unstable} minimal surface, i.e., a minimal surface contained in a critical set that is not of minimum type.
\end{thm}

In \cite[Section 8]{Morse.1939}, Morse and Tompkins provide an example of a curve $g$ for which $\Omega_{g}$ indeed contains two distinct critical sets of minimum type, so that $g$ then also spans an unstable minimal surface.

While more efficient and more general proofs for the existence of an unstable minimal surface (with respect to more natural topologies than $C^0$) have subsequently been established \cite{Struwe.1988,Dierkes.2010}, including less restrictive assumptions on the boundary curve, the original approach of Morse and Tompkins is notable
for its connections to other areas of mathematics.
As an illustration, we will now sketch a proof of \cref{thm:mountain_pass} using the previously developed machinery, starting with some intermediate results.
As a first step, we will give general formulas for \emph{numbers of births} $\alpha(t) = \sum_{q \in (t, \infty]} \m(t,q)$ and \emph{numbers of deaths} $\omega(t) = \sum_{p \in [-\infty, t)} \m(p,t)$ in q-tame persistence modules.

\begin{lem}
\label{lem:birth_death_formulas}
	Let $M$ be a q-tame persistence module.
	For each $t \in \R$ we either have
	\[
		\alpha(t) = \dim \coker (\colim_{s < t} M_{s} \to \lim_{u > t} M_{u})
	\]
	or both quantities are infinite.
	Similarly, we either have
	\[
		\omega(t) = \dim \ker (\colim_{s < t} M_{s} \to \lim_{u > t} M_{u})
	\]
	or both quantities are infinite.
\end{lem}
\begin{proof}
	None of the quantities appearing in the statement change if we replace $M$ by its radical.
	This radical admits a barcode decomposition and the claims easily follow from an explicit computation on the corresponding barcode module.
\end{proof}

If $M$ is continuous from below or above, then the colimits and limits appearing in the formulas above may simply be replaced by the constituent vector spaces of~$M$.
As a special case, this yields the following corollary.

\begin{cor}
\label{cor:regular_value_no_endpoint}
	Let $M$ be a q-tame persistence module.
	If $M$ is continuous from above and below at $t$, we have
	\[
		\alpha(t) = \omega(t) = 0.
	\]
\end{cor}

We now return to the setting of functions $F$ on metric spaces and prove some more lemmas.
To emphasize which persistence-theoretic notions are relevant, we will work with a general homology theory $\H$ such that $\H(F_{\leq \bullet})$ has certain properties, stated in the lemmas.
In all cases, \v{C}ech homology as used in \cref{thm:mountain_pass} satisfies the necessary conditions.
We start by showing that function values with non-vanishing cap numbers are indeed critical values.

\begin{lem}
\label{lem:endpoint_implies_crit_pt}
	Let $F \colon M \to \R$ be a function on a metric space with compact sublevel set filtration.
	Assume that $\H(F_{\leq \bullet})$ is q-tame and continuous from above, and consider $t \in \R$ and its cap numbers $c^{\epsilon}(t)$.
	If there exists $\epsilon > 0$ such that $c^{\epsilon}(t) > 0$, then $t$ is a critical value of $F$.
\end{lem}
\begin{proof}
	Following \cite[Remark II.6.3]{Struwe.1988}, we know that if $t$ is a regular value, there exists $\epsilon > 0$ such that the inclusion $F_{\leq s} \hookrightarrow F_{\leq t}$ is a homotopy equivalence for all $s \in [t - \epsilon, t]$.
	Thus, $\H(F_{\leq \bullet})$ is continuous from below at every regular value.
	However, we assume $\H(F_{\leq \bullet})$ to be also continuous from above at every value, and in particular at regular values.
	Hence, \cref{cor:regular_value_no_endpoint} implies that $\alpha(t) = \omega(t) = 0$ whenever $t$ is regular, which proves the claim.
\end{proof}

Next, we will analyze how the homology of sublevel sets changes at function values of critical sets of minimum type.

\begin{lem}
\label{lem:minimum_type_implies_some_birth_and_no_death}
	Let $F \colon M \to \R$ be a function on a metric space with compact sublevel set filtration and let $S$ be a critical set of minimum type with value $t$.
	Assume that~$\H$~is additive and that $\H(F_{\leq \bullet})$ is q-tame and continuous from above.
	\begin{enumerate}
		\item The number of births at $t$ satisfies $\alpha(t) \geq \dim \H(S)$.
		\item If there are no homotopically critical points with value $t$ outside $S$, then the number of deaths at $t$ satisfies $\omega(t) = 0$.
	\end{enumerate}
\end{lem}
\begin{proof}
	We start by showing that $S$ is a topological summand of $F_{\leq t}$ in the sense that $F_{\leq t}$ is homeomorphic to the disjoint union $S \sqcup (F_{\leq t} \setminus S)$.
	It suffices to show that $S$ is open and closed in $F_{\leq t}$.
	By definition, there exists a neighborhood $N$ of $\overline{S}$ in $M$ such that the function values of $F$ on $N \setminus S$ exceed $t$.
	In particular, we have $F_{\leq t} \cap N = S$, showing that $S$ is open in $F_{\leq t}$.
	Because $N$ contains $\overline{S}$, we also obtain $F_{\leq t} \cap \overline{S} = S$, showing that $S$ is closed in $F_{\leq t}$.

	Using additivity of $\H$ and \cref{lem:birth_death_formulas}, we now obtain
	\begin{align*}
		\alpha(t) &= \dim \coker \big( \colim_{s < t} \H(F_{\leq s}) \to \lim_{u > t} \H(F_{\leq u}) \big) \\
			&= \dim \coker \big( \colim_{s < t} \H(F_{\leq s}) \to \H(F_{\leq t}) \big) \\
			&= \dim \coker \big( \colim_{s < t} \H(F_{\leq s}) \to \H(F_{\leq t} \setminus S) \oplus \H(S) \big) \\
			&\geq \dim \H(S),
	\end{align*}
	where we have used the assumption that $\H(F_{\leq \bullet})$ is continuous from above for the second equality and the fact that $F_{\leq s} \subseteq (F_{\leq t} \setminus S)$ for all $s < t$ for the final inequality.

	Now, assuming that $S$ is the set of all homotopically critical points with value~$t$ implies that 
	there exists $\epsilon > 0$ such that the inclusion $F_{\leq s} \hookrightarrow (F_{\leq t} \setminus S)$ is in fact a homotopy equivalence for for all $s \in [t - \epsilon, t]$, again following \cite[Remark II.6.3]{Struwe.1988}.
	Hence, again using additivity of $\H$, continuity from above, and \cref{lem:birth_death_formulas}, we obtain that
	\begin{align*}
		\omega(t) &= \dim \ker \big( \colim_{s < t} \H(F_{\leq s}) \to \lim_{u > t} \H(F_{\leq u}) \big) \\
				&= \dim \ker \big( \colim_{s < t} \H(F_{\leq s}) \to \H(F_{\leq t}) \big) \\
				&= \dim \ker \big( \colim_{s < t} \H(F_{\leq s}) \to \colim_{s < t} H(F_{\leq s}) \oplus \H(S) \big) \\
				&= 0
	\end{align*}
	as claimed.
\end{proof}

As the last preparatory result, we will show that connectedness of $M$ implies that the $0$-th essential dimension is trivial for \v{C}ech homology.

\begin{lem}
\label{lem:essential_cech_dim}
	Let $F \colon M \to \R$ be a function on a connected non-empty metric space with compact sublevel set filtration, and let $\cp_{d}$ be the essential dimensions of its q-tame persistent \v{C}ech homology.
	Then $\cp_{0} = 1$.
\end{lem}

\begin{proof}
	Let $\H$ denote singular homology, let $p_{d}$ denote the essential dimensions of the persistent singular homology of $F_{\leq \bullet}$, and let $\CH$ denote \v{C}ech homology.
	Since $M$ is connected, we have $\dim H_{0}(M) = 1$.
	We also have $p_{d} = \dim \colim_{t \to \infty} H_{d}(F_{\leq t}) = \dim H_{d}(M)$, where the second equality holds because all the sublevel sets of $F$ are compact \cite[Proposition 3.33]{Hatcher.2002}.
	Now, the natural map from singular to \v{C}ech homology is always surjective for compact metric spaces in dimension 0 \cite{Eda.2000}, so $\cp_{0} \leq p_{0}$.
	Moreover, we clearly have $\cp_{0} \geq 1$ because $M$ is non-empty.
	In total, we obtain $1 \leq \cp_{0} \leq p_{0} = \dim H_{0}(M) = 1$, which proves the claim.
\end{proof}

We are now ready to give a proof of the Mountain Pass Theorem.

\begin{proof}[Proof of \cref{thm:mountain_pass}]
    First, we observe that $\CH_*(F_{\leq \bullet})$ is q-tame by \cref{t:local connectedness implies q-tameness} because we assume that $F_{\leq \bullet}$ is $\HLC$ for \v{C}ech homology.
    As a consequence, $\CH_*(F_{\leq \bullet})$ has a well-defined persistence diagram by \cref{t:q-tame modules have barcodes}, and thus we can consider cap numbers, births, deaths, etc.
    We write $\cc_d$, $\calpha_d$, $\comega_d$, and $\cp_d$ for the cap numbers, births, deaths, and essential dimensions of $\CH_d(F_{\leq \bullet})$, respectively.
    The sublevel set filtration of $F_{\leq \bullet}$ is compact, its persistent \v{C}ech homology is continuous from above, and $\CH$ is additive, so the previous lemmas are applicable.
	
	Now assume that $F$ has two distinct critical sets $S_{1}$ and $S_{2}$ of minimum type with values $t_{1}$ and $t_{2}$, respectively.
	Since both critical sets are non-empty, we have $\CH_{0}(S_{i}) \neq 0$, and thus the first assertion of \cref{lem:minimum_type_implies_some_birth_and_no_death} implies $\calpha_{0}(t_{i}) \geq \dim \CH_0(S_{i}) \geq 1$ for $i = 1,2$.
	This implies that there exists $\epsilon > 0$ with $\cc_{0}^{\epsilon} \geq \calpha_{0}(t_{1}) + \calpha_{0}(t_{2}) \geq 1 + 1 = 2$.
	Since $M$ is assumed to be connected, we conclude from \cref{lem:essential_cech_dim} that the essential dimension $\cp_{0}$ is $1$.
	We obtain $\comega_{0}^{\epsilon} = \cc_{0}^{\epsilon} - \cp_{0} \geq 2 - 1 = 1$.
	Thus, there must be some $t \in \R$ with $\comega_{0}(t) > 0$, so that in particular $\cc_{0}^{\epsilon}(t) > 0$.
	Thus, we may apply \cref{lem:endpoint_implies_crit_pt} to obtain that the set $S$ of homotopically critical points at value $t$ is non-empty.
	If $S$ were of minimum type, then we would have $\comega_{0}(t) = 0$ by the second assertion of \cref{lem:minimum_type_implies_some_birth_and_no_death}, contradicting the choice of $t$.
	Hence, $S$ cannot be of minimum type, which finishes the proof.
\end{proof}

\subsection{Morse's local connectivity conditions}\label{subsec:historic_hlc}

In order to complete the proof of the Unstable Minimal Surface Theorem, we explain why the Douglas functional is q-tame for \v{C}ech homology (\cref{prop:douglas_hlc}), and we discuss the issue with Morse's approach to q-tameness in \cite[Theorem 6.3]{Morse.1940}.

Throughout his work on functional topology, Morse assumed slightly varying forms of local connectivity on the resulting sublevel set filtrations in order to obtain \mbox{q-tameness}.
In particular, Morse and Tompkins used the following condition from \cite{Morse.1938,Morse.1940} in their application to minimal surface theory:
\begin{displaycquote}[p.~431]{Morse.1940}
	Let $p$ be a point of $M$ at which $F(p)=c$.
	The space $M$ is said to be \emph{locally $F$-connected} of order $r$ at~$p$ if corresponding to each positive constant $e$ there exists a positive constant $\delta$ such that each singular $r$-sphere on the $\delta$-neighborhood of $p$ and on $F_{c+\delta}$ bounds an $(r+1)$-cell of norm $e$ on $F_{c+e}$.
\end{displaycquote}
See also \cite[p.~ 25]{Morse.1938} and \cite[p.~464]{Morse.1939}, but note that the definitions given there contain evident typographical errors.
Expressed in similar language as \cref{s:connectivity}, the property of local $F$-connectedness of all orders is equivalent to the following notion, applicable to general topological spaces.

\begin{defi}
	The sublevel set filtration of a function $f \colon X \to \R$ is said to be \emph{weakly locally connected of all orders}, or \emph{weakly $\piLC$}, if for any $x \in X$, any neighborhood $V$ of $x$, and any index $t > f(x)$, there is an index $s$ with $f(x) < s < t$ and a neighborhood $U$ of $x$ with $U \subseteq V$ such that the inclusion $f_{\leq s} \cap U \hookrightarrow f_{\leq t} \cap V$ induces trivial maps on all homotopy groups.
\end{defi}

Morse then goes on to claim that the persistent \v{C}ech homology of this sublevel set filtration is q-tame, provided that $F$ is bounded from below and satisfies the assumptions of local $F$-connectivity and compactness of sublevel sets.
In the original (where the function is assumed to take values in $[0,1)$) the claim reads:
\begin{displaycquote}[Theorem 6.3, p.~432]{Morse.1940}
	Let $a$ and $c$ be positive constants such that $a < c < 1$.
	The $k^{\mathrm{th}}$ connectivity $R^k(a,c)$ of $F_a$ on $F_c$ is finite.
\end{displaycquote}
Morse does not prove this statement in the given reference, but rather refers to \cite[Theorem~6.1]{Morse.1938}.
Unfortunately, the above claim does not hold in general, as exemplified by the sublevel set filtration from \cref{e:counterexample}.
To elaborate on this, we consider a stronger version of weak local connectedness.

\begin{defi}
	The sublevel set filtration of a function $f \colon X \to \R$ is said to be \emph{weakly locally contractible}, or \emph{weakly $\LC$}, if for any $x \in X$, any neighborhood $V$ of $x$, and any index $t > f(x)$, there is an index $s$ with $f(x) < s < t$ and a neighborhood $U$ of $x$ with $U \subseteq V$ such that the inclusion $f_{\leq s} \cap U \hookrightarrow f_{\leq t} \cap V$ is homotopic to a constant map.
\end{defi}

Clearly, being weakly $\LC$ implies being weakly $\piLC$ and, if the homology $\H$ takes finite dimensional values on one-point spaces, also weakly $\HLC$.
Observe that our discussion in \cref{e:counterexample} actually establishes that the filtration given there is weakly $\LC$, so not even the weak $\LC$ condition is sufficient to ensure the q-tameness of compact sublevel set filtrations that are induced by non-continuous functions in general.
In particular, our construction invalidates Morse's claim quoted above:

\begin{cor} \label{c:counterexample}
	The function $f \colon \HE \to \R$ with value~$0$ at the origin and $1$ elsewhere defines a weakly $\LC$ compact sublevel set filtration that is not q-tame with respect to either singular or \v{C}ech homology.
\end{cor}

This counterexample reveals a gap in the argument of Morse and Tompkins, as the sublevel set filtration of $A_g$ is not actually shown to be q-tame.
Fortunately, this gap can be readily fixed by applying \cref{t:local connectedness implies q-tameness}.
This is because the proof given in \cite[Theorem 7.2, p.464]{Morse.1939} for the local connectivity of the sublevel set filtration induced by $A_g$ can actually be seen to establish a stronger property, described next.

\begin{defi}
	The sublevel set filtration of a function $f \colon X \to \R$ is said to be \emph{locally contractible} or \emph{$\LC$} if for any $x \in X$, any neighborhood $V$ of $x$ and any pair of indices $f(x) < s < t$ there is a neighborhood $U \subseteq V$ of $x$ such that the inclusion $f_{\leq s} \cap U \hookrightarrow f_{\leq t} \cap V$ is homotopic to a constant map.
\end{defi}

\begin{prop}[{\cite[p.464]{Morse.1939}}]
\label{prop:douglas_hlc}
    The sublevel set filtration of the Douglas functional $A_g \colon \Omega_g \to \R$ is $\LC$, and in particular $\HLC$.
\end{prop}

We can now finally prove the Unstable Minimal Surface Theorem.

\begin{proof}[Proof of \cref{thm:unstable_minimial_surface}]
    The sublevel set filtration of $A_g$ is compact according to \cite[p.~448]{Morse.1939}.
    Moreover, $\Omega_g$ is contractible by \cite[Theorem 4.3]{Morse.1939}.
    The sublevel set filtration of $A_g$ is also $\HLC$ for \v{C}ech homology according to \cref{prop:douglas_hlc}, so \cref{thm:mountain_pass} applies to $A_g$.
    Since any homotopcially critical point of $A_g$ corresponds to a minimal surface spanned by $g$ \cite[Theorem 6.2]{Morse.1939}, this implies the claim. 
\end{proof}

\begin{rem}
Morse introduced another condition three years earlier, which he also called local $F$-connectivity.
It roughly corresponds to being $\piLC$ with a certain added uniformity property.
In the original it reads:
\begin{displaycquote}[p.421--422]{Morse.1937}
	The space $M$ will be said to be locally $F$-connected for the order $n$ if corresponding to $n$, an arbitrary point $p$ on $M$, and an arbitrary positive constant $e$, there exists a positive constant $\delta$ with the following property.
	For $c \geq F(p)$ any singular $n$-sphere on $F \leq c$ (the continuous image on $F \leq c$ of an ordinary $n$-sphere) on the $\delta$-neighborhood $p_{\delta}$ of $p$ is the boundary of a singular $(n + 1)$-cell on $F \leq c + e$ and on $p_e$.
\end{displaycquote}
Morse also claims in the given reference that this condition is sufficient for q-tameness, but without providing a proof.
Whether this statement is true or not is not covered by our analysis, because the $\piLC$ and $\HLC$ conditions generally do not imply each other.
We expect the quoted claim to be true, but do not investigate it further.
\end{rem}

	\appendix
	
\section{Vietoris and \texorpdfstring{\v{C}}{}ech homology for compact metric spaces} \label{s:vietoris}

For greater generality, in this paper we consider \v{C}ech homology instead of the homology construction used by Morse in functional topology: metric Vietoris homology.
We now justify this choice by reminding the reader that these two constructions agree on \emph{compact} metric spaces.
We follow \cite{Dowker.1952} in our discussion; for an earlier approach, the reader may also consult \cite[Section VII.6]{Lefschetz.1942}.

First, let us recall the definition of \v{C}ech homology as presented for example by \cite[Section~IX--X]{Eilenberg.1952}.
Let $X$ be a topological space and let $\Cov(X)$ be the set of all open covers of $X$ ordered by the refinement relation, which we implicitly regard as a poset category.
Recall that for an open cover $\alpha \in \Cov(X)$ its \emph{nerve} $\Nrv(\alpha)$ is defined as the simplicial complex
\begin{equation*}
\Nrv(\alpha) =
\big\{ \beta \subseteq \alpha \mid \beta \neq \emptyset \text{ is finite and } \textstyle{\bigcap_{U \in \beta}} \, U \neq \emptyset \big\}.
\end{equation*}
For any field $\F$, the nerve construction composed with the functor of simplicial homology with coefficients in $\F$ defines a functor from $\Cov(X)$ to the category of graded $\F$-vector spaces.
The \emph{\v{C}ech homology with coefficients in $\F$} of $X$ is defined as
\begin{equation*}
\CH(X; \F) \ =
\lim_{\alpha \in \Cov(X)} H(\Nrv(\alpha); \F).
\end{equation*}
As an alternative to the nerve construction, for a cover $\alpha \in \Cov(X)$ one can define $\Vietoris(\alpha)$ as the simplicial complex
\begin{equation*}
\Vietoris (\alpha) = \left\{ \sigma \subseteq X \mid \sigma \neq \emptyset \text{ is finite and } \sigma \in U \text{ for some } U \in \alpha \right\}.
\end{equation*}
This construction is dual to the nerve construction in the sense of Dowker's Theorem \cite{Dowker.1952}, which asserts that the two complexes $\Nrv (\alpha)$ and $\Vietoris (\alpha)$ are homotopy equivalent after geometric realization.
As a consequence, we have that $H (\Nrv (\alpha); \F) \cong H (\Vietoris (\alpha); \F)$.
This isomorphism is natural with respect to refinement of covers, so we get an alternative description of \v{C}ech homology as
\begin{equation*}
\CH (X; \F) \ \cong
\lim_{\alpha \in \Cov (X)} H (\Vietoris (\alpha); \F).
\end{equation*}

If $X$ is an arbitrary metric space, this is still not exactly the same as the construction of \emph{metric} Vietoris homology, as originally defined by Vietoris \cite{Vietoris.1927} and used by Morse, which in modern notation is the limit over all covers of $X$ by metric $\delta$-balls,
\begin{equation*}
\lim_{\alpha \in \Balls(X)} H (\Vietoris (\alpha); \F),
\end{equation*}
where
\begin{equation*}
\Balls (X) = \left\{ ( B_{\delta} (x) )_{x \in X} \mid \delta > 0 \right\}
\subseteq \Cov (X).
\end{equation*}
If the metric space $X$ is compact, then for each cover $\alpha$ there exists $\lambda > 0$ by Lebesgue's number lemma \cite[Lemma 27.5]{Munkres.2000} such that $(B_{\lambda}(x))_{x \in X}$ refines $\alpha$.
In other words, if the metric space $X$ is compact, then $\Balls (X)$ is coinitial in $\Cov (X)$, that is to say, they yield the same limit.
Thus, in this case we have a natural isomorphism
\begin{equation*}
\CH (X; \F) \ \cong \,
\lim_{\alpha \in \Balls(X)} H (\Vietoris (\alpha); \F)
\end{equation*}
In other words, for compact metric spaces the metric Vietoris homology theory employed in Morse's setting is isomorphic to the \v{C}ech homology considered in this paper.
Note that the relevant isomorphisms above can all be seen to be natural with respect to continuous maps (see in particular \cite[Lemma 7a]{Dowker.1952}), so that in particular both homology constructions yield the same persistence modules when applied to sublevel set filtrations.
In the setting of compact metric spaces, the well-known defect of \v{C}ech homology not having long exact sequences for pairs of spaces also disappears if one uses field coefficients, which is necessary anyway to get persistence diagrams.

\section*{Acknowledgements}

U.B. has been supported by the German Research Foundation (DFG) through the Collaborative Research Center SFB/TRR 109 \emph{Discretization in Geometry and Dynamics}.

A.M-M. acknowledges financial support from Innosuisse grant \mbox{32875.1 IP-ICT-1} and the hospitality of the Laboratory for Topology and Neuroscience at EPFL.

M.S. has been supported by the German Research Foundation (DFG) through	the Cluster of Excellence EXC-2181/1 \emph{STRUCTURES}, and the Research Training Group RTG 2229 \emph{Asymptotic Invariants and Limits of Groups and Spaces}.
	\sloppy
	\printbibliography

@article {Courant.1941,
    AUTHOR = {Courant, Richard},
     TITLE = {Critical points and unstable minimal surfaces},
   JOURNAL = {Proc. Nat. Acad. Sci. U.S.A.},
  FJOURNAL = {Proceedings of the National Academy of Sciences of the United
              States of America},
    VOLUME = {27},
      YEAR = {1941},
     PAGES = {51--57},
      ISSN = {0027-8424},
   MRCLASS = {49.0X},
  MRNUMBER = {3484},
MRREVIEWER = {C. B. Morrey, Jr.},
       DOI = {10.1073/pnas.27.1.51},
       URL = {https://doi.org/10.1073/pnas.27.1.51},
}

@article {Morse.1925,
    AUTHOR = {Morse, Marston},
     TITLE = {Relations between the critical points of a real function of
              {$n$} independent variables},
   JOURNAL = {Trans. Amer. Math. Soc.},
  FJOURNAL = {Transactions of the American Mathematical Society},
    VOLUME = {27},
      YEAR = {1925},
    NUMBER = {3},
     PAGES = {345--396},
      ISSN = {0002-9947},
   MRCLASS = {49K10 (26B10)},
  MRNUMBER = {1501318},
       DOI = {10.2307/1989110},
       URL = {https://doi.org/10.2307/1989110},
}

@article{Cagliari.2011,
title = {Finiteness of rank invariants of multidimensional persistent homology groups},
journal = {Applied Mathematics Letters},
volume = {24},
number = {4},
pages = {516-518},
year = {2011},
issn = {0893-9659},
doi = {https://doi.org/10.1016/j.aml.2010.11.004},
url = {https://www.sciencedirect.com/science/article/pii/S0893965910004118},
author = {Francesca Cagliari and Claudia Landi}
}

@article {Schmahl.2021,
      title={Structure of semi-continuous q-tame persistence modules}, 
      author={Maximilian Schmahl},
   JOURNAL = {Homology Homotopy Appl.},
  FJOURNAL = {Homology, Homotopy and Applications},
    VOLUME = {24},
      YEAR = {2022},
    NUMBER = {1},
     PAGES = {117--128},
       DOI = {10.4310/HHA.2022.v24.n1.a6},
       URL = {https://dx.doi.org/10.4310/HHA.2022.v24.n1.a6},
}

@article {Eda.2000,
    AUTHOR = {Eda, Katsuya and Kawamura, Kazuhiro},
     TITLE = {The surjectivity of the canonical homomorphism from singular
              homology to \v{C}ech homology},
   JOURNAL = {Proc. Amer. Math. Soc.},
  FJOURNAL = {Proceedings of the American Mathematical Society},
    VOLUME = {128},
      YEAR = {2000},
    NUMBER = {5},
     PAGES = {1487--1495},
      ISSN = {0002-9939},
   MRCLASS = {55N05},
  MRNUMBER = {1712917},
MRREVIEWER = {Francisco R. Ruiz del Portal},
       DOI = {10.1090/S0002-9939-99-05670-1},
       URL = {https://doi.org/10.1090/S0002-9939-99-05670-1},
}

@book {Hatcher.2002,
    AUTHOR = {Hatcher, Allen},
     TITLE = {Algebraic topology},
 PUBLISHER = {Cambridge University Press, Cambridge},
      YEAR = {2002},
     PAGES = {xii+544},
      ISBN = {0-521-79540-0},
   MRCLASS = {55-01 (55-00)},
  MRNUMBER = {1867354},
MRREVIEWER = {Donald W. Kahn},
}

@book {Munkres.2000,
    AUTHOR = {Munkres, James R.},
     TITLE = {Topology},
      NOTE = {Second edition of [ MR0464128]},
 PUBLISHER = {Prentice Hall, Inc., Upper Saddle River, NJ},
      YEAR = {2000},
     PAGES = {xvi+537},
      ISBN = {0-13-181629-2},
   MRCLASS = {54-01},
  MRNUMBER = {3728284},
}

@book {Lefschetz.1942,
    AUTHOR = {Lefschetz, Solomon},
     TITLE = {Algebraic {T}opology},
    SERIES = {American Mathematical Society Colloquium Publications, Vol.
              27},
 PUBLISHER = {American Mathematical Society, New York},
      YEAR = {1942},
     PAGES = {vi+389},
   MRCLASS = {56.0X},
  MRNUMBER = {0007093},
MRREVIEWER = {H. Whitney},
}

@article {Polterovich.2016,
    AUTHOR = {Polterovich, Leonid and Shelukhin, Egor},
     TITLE = {Autonomous {H}amiltonian flows, {H}ofer's geometry and
              persistence modules},
   JOURNAL = {Selecta Math. (N.S.)},
  FJOURNAL = {Selecta Mathematica. New Series},
    VOLUME = {22},
      YEAR = {2016},
    NUMBER = {1},
     PAGES = {227--296},
      ISSN = {1022-1824},
   MRCLASS = {53D05 (37K05)},
  MRNUMBER = {3437837},
MRREVIEWER = {Karl Friedrich Siburg},
       DOI = {10.1007/s00029-015-0201-2},
       URL = {https://doi.org/10.1007/s00029-015-0201-2},
}

@article {Usher.2016,
    AUTHOR = {Usher, Michael and Zhang, Jun},
     TITLE = {Persistent homology and {F}loer--{N}ovikov theory},
   JOURNAL = {Geom. Topol.},
  FJOURNAL = {Geometry \& Topology},
    VOLUME = {20},
      YEAR = {2016},
    NUMBER = {6},
     PAGES = {3333--3430},
      ISSN = {1465-3060},
   MRCLASS = {53D40 (55U15)},
  MRNUMBER = {3590354},
MRREVIEWER = {Sonja Hohloch},
       DOI = {10.2140/gt.2016.20.3333},
       URL = {https://doi.org/10.2140/gt.2016.20.3333},
}

@article {Shelukhin.2019,
    AUTHOR = {Shelukhin, Egor},
     TITLE = {On the {H}ofer--{Z}ehnder conjecture},
   JOURNAL = {Ann. of Math. (2)},
  FJOURNAL = {Annals of Mathematics. Second Series},
    VOLUME = {195},
      YEAR = {2022},
    NUMBER = {3},
     PAGES = {775--839},
      ISSN = {0003-486X},
   MRCLASS = {53D40},
  MRNUMBER = {4413744},
       DOI = {10.4007/annals.2022.195.3.1},
       URL = {https://doi.org/10.4007/annals.2022.195.3.1},
}

@article {LeRoux.2018,
    AUTHOR = {Le Roux, Fr\'{e}d\'{e}ric and Seyfaddini, Sobhan and Viterbo, Claude},
     TITLE = {Barcodes and area-preserving homeomorphisms},
   JOURNAL = {Geom. Topol.},
  FJOURNAL = {Geometry \& Topology},
    VOLUME = {25},
      YEAR = {2021},
    NUMBER = {6},
     PAGES = {2713--2825},
      ISSN = {1465-3060},
   MRCLASS = {37E30 (53D05 53D40)},
  MRNUMBER = {4347306},
       DOI = {10.2140/gt.2021.25.2713},
       URL = {https://doi.org/10.2140/gt.2021.25.2713},
}

@article {Vietoris.1927,
    AUTHOR = {Vietoris, L.},
     TITLE = {\"{U}ber den h\"{o}heren {Z}usammenhang kompakter {R}\"{a}ume und eine
              {K}lasse von zusammenhangstreuen {A}bbildungen},
   JOURNAL = {Math. Ann.},
  FJOURNAL = {Mathematische Annalen},
    VOLUME = {97},
      YEAR = {1927},
    NUMBER = {1},
     PAGES = {454--472},
      ISSN = {0025-5831},
   MRCLASS = {DML},
  MRNUMBER = {1512371},
       DOI = {10.1007/BF01447877},
       URL = {https://doi.org/10.1007/BF01447877},
}

@article{Douglas.1931,
    AUTHOR = {Douglas, Jesse},
     TITLE = {Solution of the problem of {P}lateau},
   JOURNAL = {Trans. Amer. Math. Soc.},
  FJOURNAL = {Transactions of the American Mathematical Society},
    VOLUME = {33},
      YEAR = {1931},
    NUMBER = {1},
     PAGES = {263--321},
      ISSN = {0002-9947},
   MRCLASS = {53A10 (49Q20)},
  MRNUMBER = {1501590},
       DOI = {10.2307/1989472},
       URL = {https://doi.org/10.2307/1989472},
}

@article{Morse.1937,
    AUTHOR = {Morse, Marston},
     TITLE = {Functional topology and abstract variational theory},
   JOURNAL = {Ann. of Math. (2)},
  FJOURNAL = {Annals of Mathematics. Second Series},
    VOLUME = {38},
      YEAR = {1937},
    NUMBER = {2},
     PAGES = {386--449},
      ISSN = {0003-486X},
   MRCLASS = {DML},
  MRNUMBER = {1503341},
       DOI = {10.2307/1968559},
       URL = {https://doi.org/10.2307/1968559},
}

@book{Morse.1938,
    author = {Morse, Marston},
  language = {eng},
 publisher = {Gauthier-Villars},
     title = {Functional topology and abstract variational theory},
       url = {http://eudml.org/doc/192613},
      year = {1938},
}

@article{Morse.1939,
    AUTHOR = {Morse, Marston and Tompkins, C.},
     TITLE = {The existence of minimal surfaces of general critical types},
   JOURNAL = {Ann. of Math. (2)},
  FJOURNAL = {Annals of Mathematics. Second Series},
    VOLUME = {40},
      YEAR = {1939},
    NUMBER = {2},
     PAGES = {443--472},
      ISSN = {0003-486X},
   MRCLASS = {DML},
  MRNUMBER = {1503471},
       DOI = {10.2307/1968932},
       URL = {https://doi.org/10.2307/1968932},
}

@article {Morse.1943,
    AUTHOR = {Morse, Marston},
     TITLE = {Functional topology},
   JOURNAL = {Bull. Amer. Math. Soc.},
  FJOURNAL = {Bulletin of the American Mathematical Society},
    VOLUME = {49},
      YEAR = {1943},
     PAGES = {144--149},
      ISSN = {0002-9904},
   MRCLASS = {49.0X},
  MRNUMBER = {9102},
MRREVIEWER = {C. B. Tompkins},
       DOI = {10.1090/S0002-9904-1943-07879-2},
       URL = {https://doi.org/10.1090/S0002-9904-1943-07879-2},
}

@article{Morse.1940,
    AUTHOR = {Morse, Marston},
     TITLE = {Rank and span in functional topology},
   JOURNAL = {Ann. of Math. (2)},
  FJOURNAL = {Annals of Mathematics. Second Series},
    VOLUME = {41},
      YEAR = {1940},
     PAGES = {419--454},
      ISSN = {0003-486X},
   MRCLASS = {56.0X},
  MRNUMBER = {1927},
MRREVIEWER = {P. A. Smith},
       DOI = {10.2307/1969014},
       URL = {https://doi.org/10.2307/1969014},
}

@article {Eilenberg.1944,
    AUTHOR = {Eilenberg, Samuel},
     TITLE = {Singular homology theory},
   JOURNAL = {Ann. of Math. (2)},
  FJOURNAL = {Annals of Mathematics. Second Series},
    VOLUME = {45},
      YEAR = {1944},
     PAGES = {407--447},
      ISSN = {0003-486X},
   MRCLASS = {56.0X},
  MRNUMBER = {10970},
MRREVIEWER = {P. A. Smith},
       DOI = {10.2307/1969185},
       URL = {https://doi.org/10.2307/1969185},
}

@article {Azumaya.1950,
    AUTHOR = {Azumaya, Gor\^{o}},
     TITLE = {Corrections and supplementaries to my paper concerning
              {K}rull-{R}emak-{S}chmidt's theorem},
   JOURNAL = {Nagoya Math. J.},
  FJOURNAL = {Nagoya Mathematical Journal},
    VOLUME = {1},
      YEAR = {1950},
     PAGES = {117--124},
      ISSN = {0027-7630},
   MRCLASS = {09.1X},
  MRNUMBER = {37832},
MRREVIEWER = {I. Kaplansky},
       URL = {http://projecteuclid.org/euclid.nmj/1118764711},
}

@book {Eilenberg.1952,
    AUTHOR = {Eilenberg, Samuel and Steenrod, Norman},
     TITLE = {Foundations of algebraic topology},
 PUBLISHER = {Princeton University Press, Princeton, New Jersey},
      YEAR = {1952},
     PAGES = {xv+328},
   MRCLASS = {56.0X},
  MRNUMBER = {0050886},
MRREVIEWER = {H. Cartan},
       URL = {https://doi.org/10.1515/9781400877492},
}

@article{Dowker.1952,
    AUTHOR = {Dowker, C. H.},
     TITLE = {Homology groups of relations},
   JOURNAL = {Ann. of Math. (2)},
  FJOURNAL = {Annals of Mathematics. Second Series},
    VOLUME = {56},
      YEAR = {1952},
     PAGES = {84--95},
      ISSN = {0003-486X},
   MRCLASS = {56.0X},
  MRNUMBER = {48030},
MRREVIEWER = {E. H. Spanier},
       DOI = {10.2307/1969768},
       URL = {https://doi.org/10.2307/1969768},
}

@article{Barratt.1962,
    AUTHOR = {Barratt, M. G. and Milnor, John},
     TITLE = {An example of anomalous singular homology},
   JOURNAL = {Proc. Amer. Math. Soc.},
  FJOURNAL = {Proceedings of the American Mathematical Society},
    VOLUME = {13},
      YEAR = {1962},
     PAGES = {293--297},
      ISSN = {0002-9939},
   MRCLASS = {55.30},
  MRNUMBER = {137110},
MRREVIEWER = {S.-T. Hu},
       DOI = {10.2307/2034486},
       URL = {https://doi.org/10.2307/2034486},
}

@book {Milnor.1963,
    AUTHOR = {Milnor, J.},
     TITLE = {Morse theory},
    SERIES = {Based on lecture notes by M. Spivak and R. Wells. Annals of
              Mathematics Studies, No. 51},
 PUBLISHER = {Princeton University Press, Princeton, N.J.},
      YEAR = {1963},
     PAGES = {vi+153},
   MRCLASS = {57.50 (53.72)},
  MRNUMBER = {0163331},
MRREVIEWER = {H. I. Levine},
       URL = {https://doi.org/10.1515/9781400881802}
}

@article{Bott.1980,
    AUTHOR = {Bott, Raoul},
     TITLE = {Marston {M}orse and his mathematical works},
   JOURNAL = {Bull. Amer. Math. Soc. (N.S.)},
  FJOURNAL = {American Mathematical Society. Bulletin. New Series},
    VOLUME = {3},
      YEAR = {1980},
    NUMBER = {3},
     PAGES = {907--950},
      ISSN = {0273-0979},
   MRCLASS = {01A70 (49-03 57-03 58-03)},
  MRNUMBER = {585177},
MRREVIEWER = {R. S. Millman},
       DOI = {10.1090/S0273-0979-1980-14824-7},
       URL = {https://doi.org/10.1090/S0273-0979-1980-14824-7},
}

@book{Struwe.1988,
    AUTHOR = {Struwe, Michael},
     TITLE = {Plateau's problem and the calculus of variations},
    SERIES = {Mathematical Notes},
    VOLUME = {35},
 PUBLISHER = {Princeton University Press, Princeton, NJ},
      YEAR = {1988},
     PAGES = {x+148},
      ISBN = {0-691-08510-2},
   MRCLASS = {58E12 (49F10 53A10 58E05)},
  MRNUMBER = {992402},
MRREVIEWER = {Helmut Kaul},
       URL = {https://jstor.org/stable/j.ctt7zv371.8}
}

@book {Bredon.1997,
    AUTHOR = {Bredon, Glen E.},
     TITLE = {Sheaf theory},
    SERIES = {Graduate Texts in Mathematics},
    VOLUME = {170},
   EDITION = {2},
 PUBLISHER = {Springer-Verlag, New York},
      YEAR = {1997},
     PAGES = {xii+502},
      ISBN = {0-387-94905-4},
   MRCLASS = {55N30 (18F20 54B40 55-02)},
  MRNUMBER = {1481706},
       DOI = {10.1007/978-1-4612-0647-7},
       URL = {https://doi.org/10.1007/978-1-4612-0647-7},
}

@incollection {Salamon.1999,
    AUTHOR = {Salamon, Dietmar},
     TITLE = {Lectures on {F}loer homology},
 BOOKTITLE = {Symplectic geometry and topology ({P}ark {C}ity, {UT}, 1997)},
    SERIES = {IAS/Park City Math. Ser.},
    VOLUME = {7},
     PAGES = {143--229},
 PUBLISHER = {Amer. Math. Soc., Providence, RI},
      YEAR = {1999},
   MRCLASS = {53D40 (37J45 53D45 57R17 57R58)},
  MRNUMBER = {1702944},
MRREVIEWER = {David E. Hurtubise},
       DOI = {10.1016/S0165-2427(99)00127-0},
       URL = {https://doi.org/10.1016/S0165-2427(99)00127-0},
}

@article {Cohen-Steiner.2007,
    AUTHOR = {Cohen-Steiner, David and Edelsbrunner, Herbert and Harer,
              John},
     TITLE = {Stability of persistence diagrams},
   JOURNAL = {Discrete Comput. Geom.},
  FJOURNAL = {Discrete \& Computational Geometry. An International Journal
              of Mathematics and Computer Science},
    VOLUME = {37},
      YEAR = {2007},
    NUMBER = {1},
     PAGES = {103--120},
      ISSN = {0179-5376},
   MRCLASS = {68U05 (55N05)},
  MRNUMBER = {2279866},
       DOI = {10.1007/s00454-006-1276-5},
       URL = {https://doi.org/10.1007/s00454-006-1276-5},
}

@inproceedings{Chazal.2009,
	author = {Chazal, Fr\'{e}d\'{e}ric and Cohen-Steiner, David and Glisse, Marc and Guibas, Leonidas J. and Oudot, Steve Y.},
	title = {Proximity of Persistence Modules and Their Diagrams},
	year = {2009},
	isbn = {9781605585017},
	publisher = {Association for Computing Machinery},
	address = {New York, NY, USA},
	url = {https://doi.org/10.1145/1542362.1542407},
	doi = {10.1145/1542362.1542407},
	booktitle = {Proceedings of the Twenty-Fifth Annual Symposium on Computational Geometry},
	pages = {237–246},
	numpages = {10},
	location = {Aarhus, Denmark},
	series = {SCG '09}
}

@book {Dierkes.2010,
    AUTHOR = {Dierkes, Ulrich and Hildebrandt, Stefan and Sauvigny, Friedrich},
     TITLE = {Minimal surfaces},
    SERIES = {Grundlehren der Mathematischen Wissenschaften},
    VOLUME = {339},
   EDITION = {2},
 PUBLISHER = {Springer, Heidelberg},
      YEAR = {2010},
     PAGES = {xvi+688},
      ISBN = {978-3-642-11697-1},
   MRCLASS = {49-02 (49Q05 53A10 58E12)},
  MRNUMBER = {2566897},
MRREVIEWER = {Andrew Bucki},
       DOI = {10.1007/978-3-642-11698-8},
       URL = {https://doi.org/10.1007/978-3-642-11698-8},
}

@article {Crawley-Boevey.2015,
    AUTHOR = {Crawley-Boevey, William},
     TITLE = {Decomposition of pointwise finite-dimen\-sional persistence modules},
   JOURNAL = {J. Algebra Appl.},
  FJOURNAL = {Journal of Algebra and its Applications},
    VOLUME = {14},
      YEAR = {2015},
    NUMBER = {5},
     PAGES = {1550066, 8},
      ISSN = {0219-4988},
   MRCLASS = {16G20},
  MRNUMBER = {3323327},
MRREVIEWER = {M\'{a}ty\'{a}s Domokos},
       DOI = {10.1142/S0219498815500668},
       URL = {https://doi.org/10.1142/S0219498815500668},
}

@book {Oudot.2015,
    AUTHOR = {Oudot, Steve Y.},
     TITLE = {Persistence theory: from quiver representations to data
              analysis},
    SERIES = {Mathematical Surveys and Monographs},
    VOLUME = {209},
 PUBLISHER = {American Mathematical Society, Providence, RI},
      YEAR = {2015},
     PAGES = {viii+218},
      ISBN = {978-1-4704-2545-6},
   MRCLASS = {55N35 (16G20 55U10 62-07 68U05)},
  MRNUMBER = {3408277},
MRREVIEWER = {Patrizio Frosini},
       DOI = {10.1090/surv/209},
       URL = {https://doi.org/10.1090/surv/209},
}

@article {Bauer.2015,
    AUTHOR = {Bauer, Ulrich and Lesnick, Michael},
     TITLE = {Induced matchings and the algebraic stability of persistence
              barcodes},
   JOURNAL = {J. Comput. Geom.},
  FJOURNAL = {Journal of Computational Geometry},
    VOLUME = {6},
      YEAR = {2015},
    NUMBER = {2},
     PAGES = {162--191},
   MRCLASS = {16G30 (55N35 68U05)},
  MRNUMBER = {3333456},
MRREVIEWER = {Massimo Ferri},
       DOI = {10.20382/jocg.v6i2a9},
       URL = {https://doi.org/10.20382/jocg.v6i2a9},
}

@article {Chazal.2016b,
    AUTHOR = {Chazal, Fr\'{e}d\'{e}ric and Crawley-Boevey, William and de Silva,
              Vin},
     TITLE = {The observable structure of persistence modules},
   JOURNAL = {Homology Homotopy Appl.},
  FJOURNAL = {Homology, Homotopy and Applications},
    VOLUME = {18},
      YEAR = {2016},
    NUMBER = {2},
     PAGES = {247--265},
      ISSN = {1532-0073},
   MRCLASS = {55U05 (55N99)},
  MRNUMBER = {3575998},
MRREVIEWER = {Ne\v{z}a Mramor Kosta},
       DOI = {10.4310/HHA.2016.v18.n2.a14},
       URL = {https://doi.org/10.4310/HHA.2016.v18.n2.a14},
}

@book {Chazal.2016a,
    AUTHOR = {Chazal, Fr\'{e}d\'{e}ric and de Silva, Vin and Glisse, Marc and Oudot,
              Steve},
     TITLE = {The structure and stability of persistence modules},
    SERIES = {SpringerBriefs in Mathematics},
 PUBLISHER = {Springer},
      YEAR = {2016},
     PAGES = {x+120},
      ISBN = {978-3-319-42543-6},
   MRCLASS = {55N10 (16G20 55U10)},
  MRNUMBER = {3524869},
MRREVIEWER = {Henry Hugh Adams},
       DOI = {10.1007/978-3-319-42545-0},
       URL = {https://doi.org/10.1007/978-3-319-42545-0},
}

@article {Chazal.2014,
    AUTHOR = {Chazal, Fr\'{e}d\'{e}ric and de Silva, Vin and Oudot, Steve},
     TITLE = {Persistence stability for geometric complexes},
   JOURNAL = {Geom. Dedicata},
  FJOURNAL = {Geometriae Dedicata},
    VOLUME = {173},
      YEAR = {2014},
     PAGES = {193--214},
      ISSN = {0046-5755},
   MRCLASS = {55N35 (55U10)},
  MRNUMBER = {3275299},
MRREVIEWER = {Peter Bubenik},
       DOI = {10.1007/s10711-013-9937-z},
       URL = {https://doi.org/10.1007/s10711-013-9937-z},
}

@book{Polterovich.2020,
  title={Topological Persistence in Geometry and Analysis},
  author={Polterovich, L. and Rosen, D. and Samvelyan, K. and Zhang, J.},
  isbn={9781470454951},
  lccn={2019059052},
  series={University Lecture Series},
  url={https://doi.org/10.1090/ulect/074},
  doi={10.1090/ulect/074},
  year={2020},
  publisher={American Mathematical Society}
}

@book {Wilder.1949,
    AUTHOR = {Wilder, Raymond Louis},
     TITLE = {Topology of {M}anifolds},
    SERIES = {American Mathematical Society Colloquium Publications, Vol.
              32},
 PUBLISHER = {American Mathematical Society, New York, N. Y.},
      YEAR = {1949},
     PAGES = {ix+402},
   MRCLASS = {56.0X},
  MRNUMBER = {0029491},
MRREVIEWER = {E. G. Begle},
       url = {https://doi.org/10.1090/coll/032}
}

@article {Struwe.1984,
    AUTHOR = {Struwe, Michael},
     TITLE = {On a critical point theory for minimal surfaces spanning a
              wire in {${\mathbb R}\sp{n}$}},
   JOURNAL = {J. Reine Angew. Math.},
  FJOURNAL = {Journal f\"{u}r die Reine und Angewandte Mathematik. [Crelle's
              Journal]},
    VOLUME = {349},
      YEAR = {1984},
     PAGES = {1--23},
      ISSN = {0075-4102},
   MRCLASS = {58E12 (49F10 53A10)},
  MRNUMBER = {743962},
MRREVIEWER = {Autorreferat},
       DOI = {10.1515/crll.1984.349.1},
       URL = {https://doi.org/10.1515/crll.1984.349.1},
}

@article {Marques.2021,
    AUTHOR = {Marques, Fernando C. and Neves, Andr\'{e}},
     TITLE = {Morse theory for the area functional},
   JOURNAL = {S\~{a}o Paulo J. Math. Sci.},
  FJOURNAL = {S\~{a}o Paulo Journal of Mathematical Sciences},
    VOLUME = {15},
      YEAR = {2021},
    NUMBER = {1},
     PAGES = {268--279},
      ISSN = {1982-6907},
   MRCLASS = {49Q05 (49Q15 58E12)},
  MRNUMBER = {4258897},
       DOI = {10.1007/s40863-019-00159-y},
       URL = {https://doi.org/10.1007/s40863-019-00159-y},
}

@book {Dierkes.2010b,
    AUTHOR = {Dierkes, Ulrich and Hildebrandt, Stefan and Tromba, Anthony
              J.},
     TITLE = {Global analysis of minimal surfaces},
    SERIES = {Grundlehren der Mathematischen Wissenschaften [Fundamental
              Principles of Mathematical Sciences]},
    VOLUME = {341},
   EDITION = {2},
 PUBLISHER = {Springer, Heidelberg},
      YEAR = {2010},
     PAGES = {xvi+537},
      ISBN = {978-3-642-11705-3},
   MRCLASS = {49-02 (30C20 49Q05 53A10)},
  MRNUMBER = {2778928},
MRREVIEWER = {Andrew Bucki},
       DOI = {10.1007/978-3-642-11706-0},
       URL = {https://doi.org/10.1007/978-3-642-11706-0},
}

@book {Jost.1991,
    AUTHOR = {Jost, J\"{u}rgen},
     TITLE = {Two-dimensional geometric variational problems},
    SERIES = {Pure and Applied Mathematics (New York)},
      NOTE = {A Wiley-Interscience Publication},
 PUBLISHER = {John Wiley \& Sons, Ltd., Chichester},
      YEAR = {1991},
     PAGES = {x+236},
      ISBN = {0-471-92839-9},
   MRCLASS = {58E12 (32G15 49Q05 58-02 58E20)},
  MRNUMBER = {1100926},
MRREVIEWER = {Wei Yue Ding},
}

@article {Montezuma.2020,
    AUTHOR = {Montezuma, Rafael},
     TITLE = {A mountain pass theorem for minimal hypersurfaces with fixed
              boundary},
   JOURNAL = {Calc. Var. Partial Differential Equations},
  FJOURNAL = {Calculus of Variations and Partial Differential Equations},
    VOLUME = {59},
      YEAR = {2020},
    NUMBER = {6},
     PAGES = {Paper No. 188, 30},
      ISSN = {0944-2669},
   MRCLASS = {58E12 (49Q05 53C20 58E05)},
  MRNUMBER = {4160865},
MRREVIEWER = {Futoshi Takahashi},
       DOI = {10.1007/s00526-020-01853-y},
       URL = {https://doi.org/10.1007/s00526-020-01853-y},
}

@article {Jost.1990,
    AUTHOR = {Jost, J. and Struwe, M.},
     TITLE = {Morse--{C}onley theory for minimal surfaces of varying
              topological type},
   JOURNAL = {Invent. Math.},
  FJOURNAL = {Inventiones Mathematicae},
    VOLUME = {102},
      YEAR = {1990},
    NUMBER = {3},
     PAGES = {465--499},
      ISSN = {0020-9910},
   MRCLASS = {58E12 (53A10)},
  MRNUMBER = {1074484},
MRREVIEWER = {Helmut Kaul},
       DOI = {10.1007/BF01233437},
       URL = {https://doi.org/10.1007/BF01233437},
}

@article {Weinberger.2011,
    AUTHOR = {Weinberger, Shmuel},
     TITLE = {What is{$\ldots$} persistent homology?},
   JOURNAL = {Notices Amer. Math. Soc.},
  FJOURNAL = {Notices of the American Mathematical Society},
    VOLUME = {58},
      YEAR = {2011},
    NUMBER = {1},
     PAGES = {36--39},
      ISSN = {0002-9920},
   MRCLASS = {57M07 (52B55)},
  MRNUMBER = {2777589},
       URL = {https://www.ams.org/journals/notices/201101/rtx110100036p.pdf},
}
	\todos
\end{document}